\documentclass[a4paper,11pt]{amsart}
\usepackage{amsfonts}
\usepackage{amsthm}
\usepackage{amsmath}
\usepackage{amssymb}
\usepackage[final]{graphicx}
\usepackage{graphics}
\usepackage{graphicx}
\usepackage{epsfig}
\usepackage{tabularx}
\usepackage{enumerate}
\usepackage[all]{xy}
\usepackage{eepic}

\title{Topological and Ergodic properties of Symmetric sub-shifts}
\author{Rafael Alcaraz Barrera}
\address{School of Mathematics\\
The University of Manchester\\
Oxford Road, Man\-ches\-ter, M13 9PL, UK}
\email{rafael.alcarazbarrera@manchester.ac.uk} 
\date{\today}
\thanks{This paper is part of the author's doctoral dissertation. Research was supported by CONACyT scholarship for Doctoral Students no. 213600.}
\begin{document}

\begin{abstract}
The family of symmetric one sided sub-shifts in two symbols given by a sequence $a$ is studied. We analyse some of their topological properties such as transitivity, the specification property and intrinsic ergodicity. It is shown that almost every member of this family admits only one measure of maximal entropy. It is shown that the same results hold for attractors of the family of open dynamical systems arising from the doubling map with a centred symmetric hole depending on one parameter, and for the set of points that have unique $\beta$-expansion for $\beta \in (\varphi,2)$ where $\varphi$ is the Golden Ratio.
\end{abstract}
\maketitle

\section{Introduction}
\label{intro}

\noindent This paper presents studies of properties of a particular class of dynamical systems, named \textit{dynamical systems with holes}. Given a discrete dynamical system, i.e. a pair $(X,f)$, where $X$ is a compact metric space and $f:X \to X$ is a continuous transformation, with positive topological entropy and an open subset $U$, named \textit{hole}, we can consider the \textit{$U$-exceptional set} defined by $$X_U = \{ x \in X \mid f^n(x) \notin U \hbox{\rm{ for every }} n \geq 0 \hbox{\rm{ or }} n \in \mathbb{Z}\}.$$ Note that $X_U$ is compact and $f$-invariant, therefore, $(X_U, f_U)$ where $f_U = f\mid_{X_U}$ is a dynamical system on its own right. Dynamical systems of this form are called \textit{open dynamical systems, exclusion maps, maps with holes} or \textit{maps with gaps or interval exclusion systems} if $X$ is an interval \cite{bundfuss}.

\vspace{1em}The study of this class of systems have been awaken interest (see e.g \cite{dettman} and references therein), since they offer interesting questions about their properties, such as:
\begin{enumerate}
\item Is $X_U$ empty, countable or uncountable? Is $X_U$ a Cantor set of positive Hausdorff dimension?
\item Is $f_U$ transitive? Is $(X_U, f\mid_U)$ intrinsically ergodic? \cite{sidorov3}.
\end{enumerate}

\vspace{1em}Recently, there have been advances in solving Question (1) in the one dimensional case \cite{sidorov2, sidorov6, sidorov5}. Also, Bundfuss et al. in \cite{bundfuss} have asked if it is possible to classify the dynamics of interval exclusion shifts via the arithmetical properties of the boundary points of the hole. 

\vspace{1em}The idea behind this paper is to understand open dynamical systems using tools from symbolic dynamics and answer Question (2) and the mentioned question stated in \cite{bundfuss} for the doubling map with a centred, symmetric hole. Moreover, as it was stated in \cite{sidorov3}, studying the doubling map with a centred, symmetric hole is essentially the same as studying the set of points with a unique $\beta$-expansion when $\beta \in (1,2)$ (this is formalised in Section \ref{exclusion}). This fact help us to mix tools from the general theory of open dynamical systems and classical results from $\beta$-expansions. 

\vspace{1em}In Theorems \ref{transitivity} and \ref{specification} we show under whose conditions a sub-shift is transitive or has the specification property. These results allow us to show the main result of our investigation; the intrinsic ergodicity of almost every considered exclusion system (Theorem \ref{almostevery}), using the classical results stated in \cite{bowen1, parry}, as it was conjectured in \cite{sidorov3}. Also, we show in Section \ref{resultstransitivityandentropy} that the function which associate to each defining sequence the topological entropy of the associated sub-shift is a devil's staircase, which is related to the escape rate formula studied in \cite[Section 5.2]{keller}. It is worth to mention that the fractal properties of the escape rate formula are studied in \cite{demers, dettman}.

\vspace{1em} The structure of the paper is as follows. Section \ref{basic} contains the necessary background for the reader's convenience. In Section \ref{exclusion} we state well known facts about the structure of the exclusion set for the family of symmetric exclusion shifts for the doubling map. Furthermore, we show the relation of this family and the set of points who have unique $\beta$-expansion between $(0, \frac{1}{\beta-1})$ for $\beta \in [\varphi, 2)$, where $\varphi = \frac{1+\sqrt{5}}{2}$. In Section \ref{resultstransitivityandentropy} we present a way to decompose the space of parameters in terms of admissible sequences of consecutive symbols and we show the cases where the symmetric sub-shifts are transitive, and we show the behaviour of the family of symmetric shifts in terms of their topological entropy. In Section \ref{specificationsection}, we show the cases when a symmetric sub-shift satisfies the specification property. In Section \ref{resultsergodicity} we show that there exist a unique measure of maximal entropy for almost every member of this family. Finally, in Section \ref{despues}, we mention some of the obstructions to show the intrinsic ergodicity of every member of the family and some questions are posed. 

\section{Background}
\label{basic}
\noindent In this section we recall some basic tools and definitions. We refer the reader to \cite{lindmarcus} for the proofs of the statements given in this section. The work done so far is related only to shift spaces given by the alphabet $\{0,1\}$. We call its elements \textit{symbols}. A \textit{word} is a finite sequence of symbols $\omega = w_1,\ldots w_n$ where $w_i \in \{0,1\}$. Given a word $\omega$ we define its \textit{length}, $\ell(\omega)$, as the number of symbols it contains. Given two finite words $\omega = w_1,\ldots w_n$ and $\nu = u_1 \ldots u_m$ we write $\omega\nu$ to denote their \textit{concatenation}, i.e. $\omega\nu = w_1,\ldots w_n u_1 \ldots u_m$. If $\omega$ is concatenated to itself $n$ times it will be denoted by $\omega^n$. 

\vspace{1em}Let $\Sigma_2 = \mathop{\prod}\limits_{n=1}^{\infty}\{0, 1\}$ with the distance given by: $$d(x,y) = \left\{
\begin{array}{clrr}      
2^{-j} & \hbox{\rm{ if }}  x \neq y;& \hbox{\rm{where }} j = \min\{i \mid x_i \neq y_i \}\\
0 & \hbox{\rm{ otherwise.}}&\\
\end{array}
\right.$$ 

Let $\pi: \Sigma_2 \to [0,1]$ be the projection map given by $$\pi((x_i)_{i=1}^{\infty}) = \mathop{\sum}\limits_{i=1}^{\infty} \dfrac{x_i}{2^i},$$ and $\sigma: \Sigma_2 \to \Sigma_2$ the one sided shift map. As well known, $\pi$ is injective except for the set of sequences ending with $0^{\infty}$ or $1^{\infty}$, which are identified in our setting.   
 
\vspace{1em}Let $\mathcal{F}$ be a set of words: We define $$\Sigma_{\mathcal{F}} = \{x \in \Sigma_2 \mid u \hbox{\rm{ is not contained in }} x \hbox{\rm{ for any word }} u \in \mathcal{F}\}.$$ The dynamical system given by $(\Sigma_{\mathcal{F}}, \sigma\mid_{\Sigma_{\mathcal{F}}})$ is called a \textit{sub-shift}. When $\mathcal{F}$ is a finite set of words of the same length we say that $(\Sigma_{\mathcal{F}}, \sigma\mid_{\Sigma_{\mathcal{F}}})$ is a \textit{sub-shift of finite type}. A sub-shift $\Sigma \subset \Sigma_2$ is \textit{sofic} if $\Sigma$ is a factor of a sub-shift of finite type, i.e there exist a sub-shift of finite type $X$ and a semi-conjugacy $h: X \to \Sigma$ \cite[p. 185]{lindmarcus}.

\vspace{1em}We denote by $B_n(\Sigma)$ the set of admissible words of length $n$ in a sub-shift  $(\Sigma,\sigma_{\Sigma})$ and we define the \textit{language of $\Sigma$} by $$\mathcal{L}(\Sigma) = \mathop{\bigcup}\limits_{n=1}^{\infty} B_n(\Sigma).$$ Recall that the \textit{topological entropy of $(\Sigma, \sigma)$} is defined by $$h_{top}(\sigma_{\Sigma}) = \mathop{\lim}\limits_{n \to \infty} \dfrac{1}{n}\log \left|B_n(\Sigma)\right|.$$ In our setting $\log$ means $\log_2$.

\vspace{1em}Recall that given $x$ and $y \in \Sigma_2$, $x$ is \textit{lexicographically less than} $y$, denoted by $x \prec y$ if there exists $k \in \mathbb{N}$ such that $x_j = y_j$ for $i < k$ and $x_k < y_k$. This is a total order on $\Sigma_2$. Moreover, the lexicographic order $\prec$ on $\Sigma_2$ is preserved under the projection map $\pi$, i.e. $x \prec y$ if and only if $\pi(x) < \pi(y)$. We can also consider this order on $B_n(\Sigma)$, in this case $k \leq n$. Given $a, b \in \Sigma_2$ the \textit{lexicographic closed interval from $a$ to $b$} is the set $$\left[a,b\right] = \{ x \in \Sigma_2 \mid a \preccurlyeq  x \preccurlyeq b \}.$$ Similarly, we can also consider the \textit{lexicographic open interval from $a$ to $b$} by changing $\preccurlyeq$ for $\prec$.  We say that $a \in \Sigma_2$ is a \textit{finite sequence} if there exist $n$ such that for every $i \geq n$, $a_i = 0$. Also, we say that $a \in \Sigma_2$ is an \textit{infinite sequence} if there is no $n$ such that $a_i = 0$ for every $i > n$. Observe that every word $\omega$ ending with $1$ can be associated to a finite sequence by considering the sequence $\omega 0^{\infty}$. Note that given a word $\omega \in \mathcal{L}(\Sigma_2)$ starting with $1$, it can be considered as a value $a_{\omega} \in [1(0)^{\infty},1^{\infty})$. The sequence $a_{\omega}$ given by $a_i = \omega_i$ for $1 \leq i \leq \ell(\omega)$ and $a_i = 0$ for every $i > \ell(\omega)$ is called \textit{the finite sequence associated to $\omega$}. In this case, the length of $\omega0^{\infty} = \ell(\omega)$. If $x$ is a finite sequence we define its length as the number of symbols it contains before the block $0^{\infty}$. 

\vspace{1em}Given $x \in \Sigma_2$ and a word $u$ we say that \textit{$u$ occurs in $x$} or \textit{$u$ is contained in $x$} if there are coordinates $i$ and $j$ such that $u = x_i, \ldots,x_j$. For a finite sequence $x$ with length $n$ we say that $0^j$ occurs in $x$ if and only if $0^n$ occurs in $x$ in the word $x_1 \ldots x_n$ for $j < n$.

\vspace{1em}For a sequence $x \in \Sigma_2$ we define \textit{the mirror image of $x$} denoted by $\bar x$ as the sequence $\bar x = (1-x_i)_{i=1}^{\infty}$, and we denote by $\bar x_i$ to $1 - x_i$. We say that a shift space $(\Sigma,\sigma_{\Sigma})$ is \textit{symmetric} if for every $x \in \Sigma$ its mirror image $\bar x$ belongs to $\Sigma$. We are interested in a particular family of symmetric sub-shifts described below. 

\newtheorem{symmetric}[subsection]{Definition}
\begin{symmetric}
Given $a \in \Sigma_2$ such that $a_1 =1$, $$\Sigma_a = \{x \in \Sigma_2 \mid \bar a \prec \sigma^n(x) \prec a \hbox{\rm{ for every }} n \geq 0 \},$$ and $\sigma_a = \sigma\mid_{\Sigma_a}$. The sub-shifts $(\Sigma_a, \sigma_a)$ are called \textit{the symmetric family of sub-shifts parametrised by $a$}. \label{symmetric}
\end{symmetric}

It is clear that if $a \prec b$ then $\Sigma_a \subset \Sigma_b$. 

\newtheorem{finitetype}[subsection]{Proposition}
\begin{finitetype}
If $(\Sigma_{\omega}, \sigma_{\omega})$ is a symmetric sub-shift defined by a word $\omega$ ending with one, then $(\Sigma_{\omega}, \sigma_{\omega})$ is a sub-shift of finite type. \label{finitetype}
\end{finitetype}
\begin{proof}
Let $a = a_1 \ldots a_n$ be a finite sequence ending with one. Let $x \in \Sigma_a$. Then, $\overline{a_1}\overline{a_2}\ldots \overline{a_n} \prec \sigma^k(x) \prec a_1a_2\ldots a_n$ for every $k \in \mathbb{N}$, where $\overline{a_i} = 1-a_i$. Consider $$\mathcal{F} = \{w \in \mathcal{B}_n(\Sigma_2) \mid w \succ a_1a_2\ldots a_n \hbox{ or } w \prec \overline{a_1}\overline{a_2}\ldots \overline{a_n}\}.$$ This defines a finite set of forbidden words of length $n$. 
\end{proof}

\newtheorem{sofic}[subsection]{Proposition \cite[Theorem 3.5]{akiyama}}
\begin{sofic}
If $(\Sigma_a, \sigma_a)$ is a symmetric sub-shift defined by a sequence $a$ such that $a$ is pre-periodic then $(\Sigma_a, \sigma_a)$ is a sofic shift.\label{sofic}
\end{sofic}

\section{The family of exclusion shifts for the doubling map and unique $\beta$-expansions}
\label{exclusion}

\subsection*{The doubling map}

Let us consider the doubling map $f(x) = 2x \mod 1$ defined in the unit circle $S^1$. Recall that given any $x \in [0,1)$ \textit{the binary expansion of $x$} is given by $$x = \mathop{\sum}\limits_{i=1}^{\infty} \dfrac{x_i}{2^i},$$ where $x_i \in \{0, 1\}$. The digit $x_i$ is given by $x_i = 0$ if $f^i(x) \in [0,\frac{1}{2})$ and $x_i = 1$ if $f^i(x) \in [\frac{1}{2},1)$. The projection map $\pi$ is a semi-conjugacy between $2x \mod 1$ and the full one sided shift $\sigma$. Observe that $\pi$ is bi-Lipschitz where $c_1 = \frac{1}{2}$ and $c_2 = 1$ are the Lipschitz constants for $\pi$ and $\pi^{-1}$ respectively. Given $a \in [0,1]$, we will denote its \textit{binary expansion} by $\widetilde{a}$. Let $U = (a, 1-a)$ for $a \in \left[\frac{1}{3},\frac{1}{2} \right]$. In this case consider $$\Lambda_a = ([0,1]\setminus \mathop{\bigcup}\limits_{n=0}^{\infty}f^{-n}(U)) \cap \left[1-2a, 2a\right].$$ Note that if $x \in [0,a]$ then $f(x) \in [0,2a]$ and if $x \in [1-a, 1]$ then $f(x) \in [1-2a, 1]$. Then, $\Lambda_a$ is the attractor of $f\mid_{U}$. It is easy to show that if $a < \frac{1}{3}$ then $\Lambda_a = \emptyset$. Note that $\pi^{-1}(\Lambda_a)$ is a closed invariant subset of $\Sigma_2$ under the shift map. 

\newtheorem{attractor1}[subsection]{Proposition}
\begin{attractor1}
Let $a \in (\frac{1}{3}, \frac{1}{2})$. Then, $\pi^{-1}(\Lambda_a) = \Sigma_{\pi^{-1}(2a)}.$\label{attractor1}
\end{attractor1}
\begin{proof}
Let $x \in \pi^{-1}(\Lambda_a)$. Note that $\pi(x) \in \Lambda_a$, i.e. $$\pi(x) \in \left[1-2a, 2a\right] \cap (X_{(a,b)}).$$ Then, $f^n(\pi(x)) \notin (a,1-a)$ and $1-2a < f^n(\pi(x)) < 2a$ for every $n \geq 0$. By substituting $f^n(\pi(x))$ by $\pi(\sigma^n(x))$ it is clear that $\sigma^n(x) \in \Sigma_a$ for every $n \geq 0$, which implies that $$\pi^{-1}(\Lambda_a) \subset \{x \in \Sigma_2 \mid \pi^{-1}(1-2a) \prec \sigma^n(x) \prec \pi^{-1}(2a) \hbox{\rm{ for every }} n \geq 0\}.$$

Now, consider $$x \in \{x \in \Sigma_2 \mid \pi^{-1}(1-2a) \prec \sigma^n(x) \prec \pi^{-1}(2a) \hbox{\rm{ for every }} n \geq 0\}.$$ Then $2a-1 < \pi(\sigma^n(x)) < 2a$ for every $n \geq 0$. This implies that $f^n(\pi(x)) \in \left[1-2a, 2a\right]$ for every $n \geq 0$. Suppose that $\pi(x) \notin \Lambda_a$. Then, there exist $n \geq 0$ such that $f^n(\pi(x)) \in (a,1-a)$. Recall that $(a,1-a) = (a, \frac{1}{2})\cup[\frac{1}{2},1-a)$. Suppose that $f^n(\pi(x)) \in (a, \frac{1}{2})$ then $f^{n+1}(\pi(x)) \in (2a,1)$. Therefore $\sigma^{n+1}(x)) \succcurlyeq \pi^{-1}(2a)$ which contradicts that $$x \in \{x \in \Sigma_2 \mid \pi^{-1}(2a-1) \prec \sigma^n(x) \prec \pi^{-1}(2a) \hbox{\rm{ for every }} n \geq 0\}.$$ If $f^n(\pi(x)) \in [\frac{1}{2},1-a)$ then $f^{n+1}(\pi(x)) \in [0,1-2a)$. Therefore, $\sigma^{n+1}(x)) \preccurlyeq \pi^{-1}(1-2a)$ which, again, contradicts our assumption on $x$.
\end{proof}

Let $\sigma_a$ denote the restriction of $\sigma$ to $\Sigma_{a}$. We call the family $$\{(\Sigma_{\pi^{-1}(2a)}, \sigma_a) \mid a \in [\frac{1}{3},\frac{1}{2}]\}$$ \textit{the one parameter family of exclusion sub-shifts for the doubling map}. Note that $\Sigma_a$ is a symmetric sub-shift and it can be expressed by $$\Sigma_a = \{x \in \Sigma_2 \mid \sigma^n(x) \prec \pi^{-1}(2a) \hbox{\rm{ and }} 
\overline{\sigma^n(x)} \prec \pi^{-1}(2a) \hbox{\rm{ for every }} n \in \mathbb{N}\},$$ where $\pi^{-1}(2a)$ is considered to end with $0^{\infty}$ if $2a$ is a dyadic rational.

\subsection*{Unique $\beta$-expansions}
Let $1 < \beta < 2$ and $x \geq 0$. Recall that a \textit{$\beta$-expansion for $x$} is a representation of the form
$$x = \mathop{\sum}\limits_{n=1}^{\infty}b_n \beta^{-n},$$ where $b_n \in \{0,1\}$.

\vspace{1em}Recall that the \textit{$\beta$-trans\-for\-ma\-tion} is the map $\tau_{\beta}:[0,1] \to [0,1]$ given by $\tau_{\beta}(x) = \beta x \mod 1$. Observe that $\tau_{\beta}$ defines a $\beta$-expansion for $x \in [0,1]$. This expansion will be given by $b_n = [\beta \tau_{\beta}^{n-1}(x)]$ for $n \in \mathbb{N}$ and it is named \textit{the greedy expansion of $x$}. Then $X_{\beta} \subset \Sigma_2$ is the set of all sequences obtained this way. Note that $\sigma_{\beta} = \sigma\mid_{X_{\beta}}$ is conjugated to $\tau_{\beta}$ \cite{sidorov3}. Let $a = (a_i)_{i=1}^{\infty}$ be the greedy expansion of $1$. In \cite{parry2}, Parry has shown that $$X_{\beta} = \{x \in \Sigma_2 \mid \sigma^n(x) \prec (a_i)_{i=1}^{\infty}\}.$$ $(X_{\beta}, \sigma_{\beta})$ where $\sigma_{\beta} = \sigma\mid_{X_{\beta}}$ is \textit{the usual $\beta$ shift}. 
\vspace{1em}It is easy to show that for every $x \in [0, \frac{1}{\beta-1}]$ there is at least one expansion for $x$ obtained in a similar way to the doubling map. $0$ and $\frac{1}{\beta-1}$ are avoided since their expansions are $0^{\infty}$ and $1^{\infty}$ respectively. Let $x \in (0, \frac{1}{\beta -1})$ and the following multivalued map described as follows: $$T_{\beta} = \left\{
\begin{array}{clrr}      
\beta x, & \hbox{\rm{ if }}  x \in [0, \frac{1}{\beta}];\\
\beta x, \hbox{\rm{ or }} \beta x - 1 & \hbox{\rm{ if }} x \in (\frac{1}{\beta}, \frac{1}{\beta(\beta - 1)});\\
\beta x - 1, & \hbox{\rm{ if }} x \in [\frac{1}{\beta(\beta - 1)}, \frac{1}{\beta-1}].\\
\end{array}
\right.$$

Observe that if $x \in (\frac{1}{\beta}, \frac{1}{\beta(\beta - 1)})$ we have two choices of values. The region $x \in (\frac{1}{\beta}, \frac{1}{\beta(\beta - 1)})$ is called \textit{the switch region}. It is clear that $0$ and $\frac{1}{\beta - 1}$ have unique expansions given by $0^{\infty}$ and $1^{\infty}$ respectively. In \cite{erdos} Erd\"os et al. shown that for every $\beta < \varphi$ any $x \in (0, \frac{1}{\beta-1})$ has uncountably many $\beta$-expansions. Nonetheless, in \cite[Theorem 2]{sidorov2} (Theorem \ref{teo1} on Section \ref{resultstransitivityandentropy}) the authors show the set of points with unique $\beta$-expansion in $(0, \frac{1}{\beta-1})$, $\widetilde{X_{\beta}}$, is non-empty if $\beta > \varphi$ where $\varphi$ is the golden ratio. We say that a $\beta$-expansion of $x$, $b = (b_n)_{n=1}^{\infty}$ is \textit{lazy} if $\bar b$ is greedy. Let $\beta \in (\varphi, 2)$. \cite[Lemma 2]{erdos} gives the necessary conditions of a sequence $(b_n)_{n=1}^{\infty}$ to be a unique $\beta$-expansion of $x \in \widetilde{X_{\beta}}$, that is $(b_n)_{n=1}^{\infty}$ is greedy and lazy. Also, by \cite[Theorem 1]{erdos}, we assure that a sequence $(b_n)_{n=1}^{\infty} \in \Sigma_{2}$ is a unique expansion of $1$ if then $(b_n)_{n=1}^{\infty}$ is greedy and lazy. 

\vspace{1em}Let $\beta \in (\varphi,2)$ and $\Sigma_{\beta} \subset \Sigma_2$ the set of sequences which are unique $\beta$-expansions for $x \in (0, \frac{1}{\beta-1})$. Let $\pi_{\beta}: \Sigma_{\beta} \to \widetilde{X}_{\beta}$ the projection given by $$\pi_{\beta}(x)  = \mathop{\sum}\limits_{n=1}^{\infty}b_n \beta^{-n}.$$ Observe that if $x_1 = 0$, $\pi_{\beta} \in [0, \frac{1}{\beta}]$ and if $x_1 = 1$, $\pi_{\beta} \in [\frac{1}{\beta(\beta - 1)}, \frac{1}{\beta - 1}]$. Then $\widetilde{X}_{\beta}$ has empty intersection with the switch region \cite{sidorov4}. In \cite{parry2} it was shown that for every $x \leq 1$ the greedy expansion of $x$, $(b_n)_{n=1}^{\infty}$, satisfies that $\sigma^k((b_n))_{n=1}^{\infty} \prec (d_n)_{n=1}^{\infty}$ for every $k \geq 0$, where $(d_1)_{n=1}^{\infty}$ is the greedy expansion of $1$. Then if $(b_n)_{n=1}^{\infty}$ is unique, $\overline{\sigma^k(b_n)_{n=1}^{\infty}} = \sigma^k(\bar b_n)_{n=1}^{\infty} \prec (d_n)_{n=1}^{\infty}$. This implies that $\Sigma_{\beta}$ is $\sigma$-invariant and it can be represented symbolically as $$\Sigma_{\beta} = \{x \in \Sigma_2 \mid \sigma^n(x) \prec (d_i)_{i=1}^{\infty} \hbox{\rm{ and }} 
\overline{\sigma^n(x)} \prec (d_i)_{i=1}^{\infty} \hbox{\rm{ for every }} n \in \mathbb{N}\}.$$

\vspace{1em}Note that $\Sigma_{\beta}$ and $\Sigma_a$ are essentially the same sub-shift. Let $\varphi < \beta < 2$. Given the unique $\beta$-expansion of $1$, $(d_n)_{n=1}^{\infty}$, there exist $a \in (\frac{1}{3}, \frac{1}{2})$ such that $\pi^{-1}(2a) = (d_n)_{n=1}^{\infty}$. 

\vspace{1em}We say that a sequence $a$ is a \textit{Parry symmetric sequence} if $a$ is the unique $\beta$-expansion of $1$ for some $\beta \in (\varphi,2)$ and we denote the set of Parry's symmetric sequences by $P$. Therefore, by \cite[Theorems 3.7, 3.8]{nilsson} we can assure that $\pi(P)$ is a set of Lebesgue measure zero with $\dim_H(\pi(P)) = 1$. Denote by $N = \Sigma^1_2 \setminus P$ where $\Sigma^1_2$ is the set of sequences such that $x_1 = 1$. Note that, by \cite[Proposition 4.1]{bundfuss} we can assure that if $\pi^{-1}(2a) \in N$ then $\Sigma_a$ is a sub-shift of finite type. For every $a \in N$ consider $$n_a = \mathop{\min}\{n \in \mathbb{N} \mid \sigma^n(a) \succ a\}.$$ Note that $n_a$ exists for every $a \in N$. Then we define $\varsigma: \Sigma^1_2 \to P$ by $$\varsigma(a) = \left\{
\begin{array}{clrr}      
a & \hbox{\rm{ if }}  a \in P\\
(a_1 \ldots a_{n_a})^{\infty} & \hbox{\rm{ otherwise.}}&\\
\end{array}
\right.$$

Observe that $\varsigma$ is well defined, since $n_a$ exists for every $a \in N$ and it is unique. Also, $\varsigma$ is surjective but not injective. The following lemma shows that the sub-shift corresponding to sequence $a \in N$ have the same dynamical properties as a sub-shift defined by a Parry symmetric sequence.

\newtheorem{lemabeta}[subsection]{Lemma}
\begin{lemabeta}
For every $a \in \Sigma_2$ starting with $1$, $\Sigma_{a} = \Sigma_{\varsigma(a)}$.\label{lemabeta}
\end{lemabeta}
\begin{proof}
Note that if $a \in P$ then $\Sigma_a = \Sigma_{\varsigma(a)}$. Consider $a \in N$. Note that $\varsigma(a) \prec a$, then $\Sigma_{\varsigma(a)} \subset \Sigma_a$. Assume that there exists $x \in \Sigma_a  \setminus \Sigma_{\varsigma(a)}$. This implies that $x_i = \varsigma(a)_i$ at least for every $1 \leq i \leq n_a$. Let $k = \min \{j \geq n_a \mid \varsigma(a)_j < x_j\}$. Observe that $x_k \leq a_j$. Since $x \in \Sigma_a$ then $\sigma^n(x) \prec a$ for every $n \geq 0$. Consider $\sigma^{n_a}(x)$. Note that $\sigma^{n_a}(x)_{j-n_a} \succ \varsigma(a)_{j-n_a}$. This implies that $\sigma^{n_a}(x)_{j-n_a} \succ a_{j-n_a}$, that is $x \notin \Sigma_a$ which is a contradiction. 
\end{proof}

\section{Transitivity and the Entropy function}
\label{resultstransitivityandentropy}

\subsection*{Transitivity}
\noindent In this section we place particular emphasis in showing the arithmetic conditions of the defining sequence $\omega$ in order to have a transitive sub-shift of finite type (Theorem \ref{transitivity}) and we give a lower bound for this to occur (Theorem \ref{critical}). Recall that a sub-shift $(\Sigma, \sigma_{\Sigma})$ is \textit{(topologically) transitive} or \textit{irreducible} if for every ordered pair of words $u,v \in \mathcal{L}(\Sigma)$ there is a $w \in \mathcal{L}(\Sigma)$, called a \textit{bridge}, such that $uwv \in \mathcal{L}(\Sigma)$ \cite{lindmarcus}.The following result gives us a lower bound for the sequences whose define symmetric sub-shifts.

\newtheorem{teo1}[subsection]{Theorem \cite[Theorem 2]{sidorov2}}
\begin{teo1}
Let $a \in [1(0)^{\infty},1^{\infty})$. Then:
\begin{enumerate}[i)]
\item $\Sigma_a$ is empty if $a \in [1(0)^{\infty},(10)^{\infty}]$;
\item $\Sigma_a = \{(01)^{\infty},(10)^{\infty}\}$ if $a \in ((10)^{\infty}, (1100)^{\infty}]$;
\item $\Sigma_a$ is countable $a \in ((1100)^{\infty},a^*)$;
\item $\Sigma_a$ is uncountable $a \in [a^*,1^{\infty})$.
\end{enumerate}
\label{teo1}
\end{teo1}

Here, $a^*$ is the image under the shift map of the \textit{Thue-Morse sequence}. 

\vspace{1em}Note that for a finite sequence $a$, the shift $(\Sigma_a, \sigma_a)$ admits a \textit{maximal (respectively minimal) word of length $\ell(a)$} given by $a_{\max} = a_1 \ldots a_{\ell(a) - 1}0$ and $a_{\min} = \bar a_{max}$. Moreover, for every $1 \leq n < \ell(a)$, the words $a_{\max_n} = a_1 \ldots a_n$ is the \textit{maximal word of length $n$} and $\bar a_{\max_n}$ is the \textit{minimal word of length $n$}. 

\newtheorem{consecutive}[subsection]{Proposition}
\begin{consecutive}
Let $n \geq 2$. For any $a \in (1^n,1^{n+1}]$, if $x \in \Sigma_a$ then $0^{n+1}$ and $1^{n+1}$ does not occur in $x$. \label{consecutive}
\end{consecutive}
\begin{proof}
Let $n \geq 2$ and $a \in (1^n,1^{n+1}]$. Since $\Sigma_a$ is symmetric, it contains the sub-shift given by $a^{\prime} = 1^n$, then the words $0^n$ and $1^n$ are admissible words of length $n$ for $\Sigma_a$. Suppose that there exists a point $x \in \Sigma_a$ such that the block $1^{n+1}$ occurs. Let $k \in \mathbb{N}$ be such that $[x_{k}, x_{k+n+1}] = 1^{n+1}$. Consider $\sigma^{k}(x)$. Notice that $\sigma^{k}(x) \succcurlyeq 1^{n+1}$, which is a contradiction. By symmetry we can conclude the result.
\end{proof}

The following Theorem shows that not every symmetric sub-shift is transitive. 

\newtheorem{critical}[subsection]{Theorem}
\begin{critical}
For any $a \in \left((10)^{\infty}, 11(01)^{\infty} \right)$, $\Sigma_a$ is not transitive.\label{critical}
\end{critical}
\begin{proof}
Consider $a \in ((10)^{\infty}, 11(01)^\infty)$. Note that $a$ cannot have three consecutive zeros or ones because $11(01)^\infty \prec 111$. Note that $a = 11(01)^{k}001\ldots$ for some $k \geq 1$. This implies that the blocks $11(01)^{k+1}$ and $00(10)^{k+1}$ are not allowed. Observe that the periodic orbit $(01)^{\infty} \in \Sigma_a$. This implies that the block $(10)^k$ is admissible for every $k \in \mathbb{N}$. Consider the words $u=11$ and $v=(10)^{k+1}$. Note that $u$ and $v$ can not be concatenated. Besides, they can not be connected by $1$ because $u1v$ will have three consecutive ones, and they can not be connected by $0$ because the block $[u0v]$ contains the block $[11(01)^{k+1}]$. Note that $u$ and $v$ can not be connected by $00$ (and by any word ending in $00$) because the block $[00(10)^{k+1}]$ will occur. Furthermore, $u$ and $v$ can not be connected by $01$ (and by any word ending in $01$ or $1$) because the block $[11(01)^{k+1}]$ will occur. Therefore the only possibility is to consider a word $w$ ending in $10$. Note that if $w_{n-1}w_n = 10$, then $w_{n-3}w_{n-2} = 01$ by the argument shown above. Then, $u$ and $v$ can not be connected. This shows that $\Sigma_a$ is not transitive.
\end{proof}

For any finite sequence $\omega$, we define $t(\omega)$ to be $$\omega  \stackrel{t}{\longrightarrow} w_1\ldots w_{\ell(\omega)} \bar w_1 \ldots \bar w_{\ell(\omega)-1}1.$$ If $\omega = 10^{\infty}$ then $\omega  \stackrel{t}{\longrightarrow} 101$. We denote by $\omega^{\prime}$ to $t(\omega)$. In addition, let

\begin{itemize}
 \item [$i)$] $\omega^{\prime \prime} = \mathop{\lim}\limits_{n \to \infty} t^n(\omega)$; 
 \item [$ii)$] $\omega^{\prime \prime \prime} = w_1\ldots w_{\ell(\omega)} (\bar w_1 \ldots w_{\ell(\omega)-1}1)^{\infty}$;
\end{itemize}

\vspace{1em} We define $1^{\prime \prime \prime}$ to be $1(01)^{\infty}$. Note that $\omega^{\prime \prime}$ coincides with $a^*$, if $\omega = 11$. We will name $\omega^{\prime \prime}$ the \textit{generalised Thue-Morse sequence of $\omega$}.

\vspace{1em}From now on, we may assume that $\omega > a^*$ in order to have $h_{top}(\Sigma_{\omega}) > 0$.

\newtheorem{irreducibleword}[subsection]{Definition}
\begin{irreducibleword}
\normalfont{We say that a finite sequence starting with $1$ $\omega \in \mathcal{L}(\Sigma_2)$ is \textit{irreducible} if for any $k < \ell(\omega)$ such that $w_k = 1$ the word $\omega_k = w_1, \ldots, w_k$ satisfies that $\omega_k^{\prime \prime \prime} \prec \omega$.}\label{irreducibleword}
\end{irreducibleword}

Note that $11$ is an irreducible word because of the way we define $1^{\prime \prime \prime}$. It is not true that for any $k \leq \ell(\omega)$ such that $w_k = 1$ the word $\omega_k^{\prime \prime \prime}$ belongs to $\Sigma_{\omega}$. Observe that if $\omega$ is a finite sequence ending with $1$, then for every $k \in \mathbb{N}$ and for every $u \in B_k(\Sigma_{\omega})$ there exists $j > k$ such that the word $u_1 \ldots u_k(u_{k+1} \ldots u_{j-1}1) \in B_j(\Sigma_{\omega})$. This is a direct consequence of the symmetry of the sub-shift and \cite[Theorem 1]{erdos}.

\newtheorem{remark2}[subsection]{Lemma}
\begin{remark2}
Let $\omega$ be a finite sequence starting with $1$. Then for $v \in \mathcal{L}(\Sigma_{\omega})$ such that $v_1 = 0$, the word $1v \in \mathcal{L}(\Sigma_{\omega})$. \label{remark2}
\end{remark2}
\begin{proof}
Let $v \in \mathcal{L}(\Sigma_{\omega})$ and consider the word $1v$. Note that for every $1 \leq j \leq \ell(v)$, $\sigma^j(1v) \in  \mathcal{L}(\Sigma_{\omega})$. By Theorem \ref{teo1} $\omega_i = 1$ for every $i \in {1, \ldots n}$ for some $2 \leq n \leq \ell(\omega)$. This implies that $1v \prec \omega$. By symmetry $\bar \omega \prec 1v$.
\end{proof}

Let $\omega \in \left[11(01)^{\infty}, 1^{\infty} \right]$. Note that by Definition \ref{symmetric} and Lemma \ref{remark2} it is true that for any $v \in \mathcal{L}(\Sigma_{\omega})$ such that $v_1 = 1$ the word $0v \in \mathcal{L}(\Sigma_{\omega})$. Therefore, Lemma \ref{remark2} also implies that for any word starting with $0$ (respectively $1$) the word $1(01)^kv$ is admissible (respectively $0(10)^kv$) for any $k \in \mathbb{N}$. Also, note that for every $v\in \mathcal{L}(\Sigma_{\omega})$ such that $v_1 = 1$. Proposition \ref{consecutive} and an adapted proof of Lemma \ref{remark2} implies that the word $0^{n-1}v$ is admissible if $\omega \in \left[1^n, 1^{n+1}\right]$. Moreover, note that it is not always possible to concatenate a finite word $u$ ending with $1$ with $(01)^k$ gluing it to the left, i.e. $u(01)^k$. Next, we show that the sub-shift given by an irreducible word $\omega$ is transitive. The presented proof give us an upper bound for the specification number of a symmetric sub-shift of finite type defined in Section \ref{specificationsection}. This bound plays an important role in the results presented in Section \ref{specificationsection}.

\newtheorem{transitivity}[subsection]{Theorem}
\begin{transitivity}
If $\omega$ is an irreducible word then $\Sigma_{\omega}$ is a transitive sub-shift of finite type.\label{transitivity}
\end{transitivity}
\begin{proof}
It suffices to show that for any $u, v \in \mathcal{L}(\Sigma_{\omega})$ such that $\ell(u), \ell(v) \leq \ell(\omega)$ there exists $w$ such that $uwv \in \mathcal{L}(\omega)$. Note that by Lemma \ref{remark2} just need to consider words $u$ ending with $1$ and words $v$ starting with $1$. Moreover, Lemma \ref{remark2} also imply that we just need to consider $u, v \in \mathcal{L}(\sigma_{\omega})$ such that $\ell(u) = \ell(v) = \ell(\omega)$.

\vspace{1em}Let $\omega = w_1 \ldots w_{\ell(\omega)}$ be an irreducible word such that $\omega \in \left[1^n, 1^{n+1}\right)$ for some $n \geq 2$. Note that if the blocks $0^n$ or $1^n$ do not occur in $u$ then $u$ and $v$ can be connected by $0^{n-1}$. 

\vspace{1em}Let $u, v \in B_{\ell(\omega)}(\Sigma_{\omega})$ such that $0^n$ and $1^n$ do not occur simultaneously on $u$. Consider $$j = \min \{k \in \{0, \ldots \ell(\omega)-n \} \mid \sigma^k(u) = 0^n1 \ldots u_{\ell(\omega) - 1}1 \hbox{\rm{ or }} 1^n0 \ldots u_{\ell(\omega)-1} 1 \}.$$ Note that $\sigma^j(u) \in B_{\ell(\omega)-j}(\Sigma_{\omega})$. Suppose that $\sigma^j(u)$ starts with $0^n$. If $\sigma^j(u) \succ \omega_{\min_{\ell(\omega)-j}}$, then $0^{n-1}$ is a bridge between $u$ and $v$. If $\sigma^j(u) \preccurlyeq \omega_{\min_{\ell(\omega)-j}}$, let $z = z_1, \ldots, z_j$ be such that $z_i = \omega_{{\min_{\ell(\omega)}}_{j+i}}$ for $i \in \{1, \ldots , j\}$. Then $\sigma^j(u)z = \omega_{\min}$. Note that $[\omega_{\min}0^{n-1}]$ is always an admissible block. Then the word $w = z0^{n-1}$ is a bridge between $u$ and $v$. Observe that the argument is similar if $\sigma^j(u)$ starts with $1^n$. In such a case $z = z_1, \ldots, z_j$ will be such that $z_i = \omega_{{\max_{\ell(\omega)}}_{j+i}}$. Then the word $w = z10^{n-1}$ is a bridge between $u$ and $v$. If $0^n$ and $1^n$ occur simultaneously in $u$ and $u_{\ell(u) - j} \neq 1$ for every $j \in \{0, \ldots n_1\}$ consider $$j = \max \{j \in \{0, \ldots \ell(\omega)-2n \} \mid \sigma^k(u) = 0^n1 \ldots u_{\ell(\omega) - 1}1 \hbox{\rm{ or }} 1^n0 \ldots u_{\ell(\omega)-1} 1 \}.$$ Then we can construct $w$ in a similar way as explained above. Suppose $u_{\ell(u) - j} \neq 1$ for every $j \in \{0, \ldots n_1\}$. Then, let $$k = \mathop{\max}\limits_{1 \leq j < \ell(\omega)} \{ \omega_k = 1 \hbox{\rm{ and }}w_1, \ldots w_k \hbox{\rm{ is irreducible }}\}.$$ Let $\omega_k = w_1, \ldots w_k$. Note that such $k$ exists because $11(01)^{\infty} \leq \omega$. Observe that $\bar \omega_k^{n-1}$ is admissible and $1^n \bar \omega_k^{n-1}0^{n-1}$ is an admissible word. Then $\bar \omega_k0^{n-1}$ is a bridge from $u$ to $v$.
\end{proof}

Recall that a sub-shift $(\Sigma, \sigma_{\Sigma})$ is \textit{(topologically) mixing} if for every ordered pair of words $u, v \in \mathcal{L}$ there is $N \in \mathbb{N}$ such that for each $n \geq N$ there is a $w \in B_n(\Sigma)$, such that $uwv \in \mathcal{L}(\Sigma)$ \cite{lindmarcus}. As a consequence of \cite[Proposition 4.5.10 (4)]{lindmarcus} and \cite[Proposition 2.10, Proposition 2.16]{sidorov4} we can obtain \cite[Theorem 46]{nilsson1} which states that $(\Sigma_{\omega}, \sigma_{\omega})$ is mixing if $\omega$ is irreducible.

\subsection*{The Entropy function structure} Recall that a function $f:\left[a,b\right] \to \mathbb{R}$, where $\left[a, b\right]$ is a \textit{Cantor's function, singular function} or a \textit{devil's staircase} if $f$ satisfies the following properties: 
\begin{enumerate}
 \item $f$ is continuous;
 \item $f(a) < f(b)$ (or $f(a) > f(b)$);
 \item $f$ is non-decreasing (or non-increasing) on $\left[a,b\right]$;
 \item There exists a set $N$ of Lebesgue measure $0$ such that for all $x \in \left[a,b\right] \setminus N$ the derivative of $f$ in $x$ exists and is zero.
\end{enumerate}

For every $a \in \left[(10)^{\infty}, 1^{\infty}\right]$, it is natural to associate to each $a$ the topological entropy of the symmetric sub-shift $\Sigma_a$, as well as associate to each $a \in [\frac{1}{3}, \frac{1}{2}]$ the entropy of $f$ restricted to the attractor $\Lambda_a$. As a consequence of \cite[Theorem 4]{urbanski2} we assure that the entropy function is continuous. During this section we show that in fact, the entropy function associated to the attractor of the doubling map with a symmetric holes is a devil's staircase and we provide a characterisation of the end points of the entropy plateaus (Theorem \ref{plateaus2}). Moreover, we characterise the exceptional set $\mathcal{E}$ using approximations by sub-shifts of finite type given by $i$-irreducible words (Theorem \ref{exceptional1}). We say that a lexicographical interval $\left[a,b\right]$ is an \textit{entropy plateau}, if $h_{top}(\Sigma_c) = h_{top}(\Sigma_a)$ for every $c \in \left[a,b\right]$ and $h_{top}(\Sigma_c) \neq h_{top}(\Sigma_a)$ if $c \notin \left[a,b\right]$.

\newtheorem{interval1}[subsection]{Proposition}
\begin{interval1}
Let $\omega, \upsilon$ be finite sequences and suppose that $\upsilon \in (\omega, \omega^{\prime \prime \prime}).$ Then $(\upsilon, \upsilon^{\prime \prime \prime}) \subset (\omega, \omega^{\prime \prime \prime}).$ \label{interval1} 
\end{interval1}
\begin{proof}
Let $\omega, \upsilon$ be finite sequences satisfying $\upsilon \in (\omega, \omega^{\prime \prime \prime})$. It suffices to show that $\upsilon^{\prime \prime \prime} \prec \omega^{\prime \prime \prime}$. Note that $v_i = w_i$ for every $i \in \{1, \ldots, \ell(\omega)\}$. Firstly, suppose that $$\upsilon = w_1 \ldots w_{\ell(\omega)}(\bar w_1 \ldots w_{\ell(\omega)-1} 1)^k.$$ Then $\upsilon^{\prime \prime \prime }_{(k+2)\ell(\omega)} = 0$ and $\omega^{\prime \prime \prime}_{(k+2)\ell(\omega)} = 1$. Note that this also implies that $\upsilon^{\prime \prime \prime} \prec \omega_{k+1}$. Hence, $\upsilon^{\prime \prime \prime} \prec \omega^{\prime \prime \prime}$. 

Note that the sequence $\{ \omega_k \}_{k \geq 0}$ given by $$\omega_k = w_1 \ldots w_{\ell(\omega)}(\bar w_1 \ldots w_{\ell(\omega)-1} 1)^k$$ satisfies:
 
 \begin{enumerate} 
 \item $\omega_k \prec \omega_{k+1}$ for every $k \geq 0$ and $\mathop{\lim}\limits_{k \to \infty} \omega_k = \omega^{\prime \prime \prime}$;
 \item $(\omega_k,\omega_{k}^{\prime \prime \prime}) \cap (\omega_{k+1},\omega_{k+1}^{\prime \prime \prime}) = \emptyset.$
 \item $\omega^{\prime \prime} \in \left[\omega^{\prime}, \omega_2\right]$.
 \end{enumerate}
 
Suppose now that $\upsilon \in (\omega_k, \omega_{k+1})$ for some $k \geq 0$. It suffices to show that $\upsilon^{\prime \prime \prime} \leq \omega_{k+1}^{\prime \prime \prime}$. Note that $\ell(\upsilon) > \ell(\omega_k)$ and $v_i = w^k_i$ for every $i \in \{1, \ldots \ell(\omega_k)\}$ where $w^k_i$ is the $i$-th digit of $\omega_k$. In fact, $\ell(\upsilon) \geq \ell(\omega_k)+r$, where $r = \mathop{\min}\{1 \leq r \leq \ell(\omega) \mid  \omega_r = 0\}$. Let $$S = \{j \in \{\ell(\omega_k)+r \ldots \ell{\omega_{k+1}}\} \mid \omega^{k+1}_j = 1 \}.$$ 

If $\ell(\upsilon) > \ell(\omega_{k+1})$ then there exist $j \in S$ such that $\upsilon_j = 0$, then it is obvious that $\upsilon^{\prime \prime \prime} \prec \omega^{\prime \prime \prime}_{k+1}$. If $\ell(\upsilon) \leq \ell(\omega^{k+1})$ we have two cases to consider. If $\upsilon$ is not a truncated word: i.e. if there exists $j \in S$ such that $\upsilon_j = 1$, $\upsilon_{j+1} = 0$, $j+1 \in S$ and $\ell(\upsilon) > j+1$ then the result holds. Suppose that $\upsilon$ is a truncated word: i.e. there exist $j \in S$ such that $\upsilon_j = 1$ and $\ell(\upsilon) = j$. If $j+1 \in S$ then $\upsilon^{\prime \prime \prime} \prec \omega_{k+1}$. Suppose that $j+1 \notin S$. Then $\upsilon^{\prime \prime \prime} = \omega^{k+1}_1 \ldots \omega^{k+1}_j(\bar \omega^{k+1}_1 \ldots \bar \omega^{k+1}_{j-1}1)^{\infty}$. Observe that in this case, it suffices to show that $\bar \omega_1 \ldots \bar \omega_j \leq (\bar \omega_{j+1}\ldots \bar \omega_{\ell(\omega)} \bar \omega_1 \ldots \bar \omega_{j})^{\infty}$, which is a consequence of Theorem \cite[Theorem 1]{erdos}.
\end{proof}

For every $k \in \mathbb{N}$ consider $\omega_{a^*_k} = a^*_1 \ldots a^*_{2^k}$. A sequence $a$ is an \textit{$i$-sequence} if ${\omega_{a^*_{i+1}}}^{\prime \prime \prime} \prec a \preccurlyeq {\omega_{a^*_i}}^{\prime \prime \prime}$. For $i = 0$, $a$ is $0$-sequence if it satisfies that $\omega_{a^*_1} \prec a \prec 1^{\infty}$. A finite $i$-sequence $\omega$ is an \textit{$i$-irreducible} if for every $2^i < j < \ell(\omega)$ such that $w_j = 1$, $(w_1...w_j)^{\prime \prime \prime} \prec \omega$. Note that $0$-irreducible sequences are simply irreducible sequences.  

\newtheorem{powersoftwo}[subsection]{Lemma}
\begin{powersoftwo}
For every $k \in \mathbb{N}$, $$h_{top}(\Sigma_{\omega_{a^*_k}^{\prime \prime \prime}}) = \dfrac{1}{2^k}.$$\label{powersoftwo} 
\end{powersoftwo}
\begin{proof}
Let $k \in \mathbb{N}$. Recall that $\omega_{a^*_k} \prec a^*$. By Theorem \ref{teo1}, $\Sigma_{\omega_{a^*_k}}$ is a countable set. Nonetheless $a^* \prec \omega_{a^*_k}^{\prime \prime \prime}$, then $h_{top}(\Sigma_{\omega_{a^*_k}^{\prime \prime \prime}}) > 0$. Henceforth the entropy of $\Sigma_{\omega_{a^*_k}^{\prime \prime \prime}}$ will be given by $$A = \{x \in \Sigma_{\omega_{a^*_k}^{\prime \prime \prime}} \mid \omega_{a^*_k} \hbox{\rm{ occurs in }} x \}.$$ Take $x \in A$. Without loss of generality it can be assumed that $x_i = \omega_{a^*_k}i$ for every $i \in \{1 \ldots 2^k\}.$ Note that $x_i = \bar \omega_{a^*_n}i$ for every $i \in \{2^k + 1, \ldots, 2^{n+1}-1\}$ and $x_{2^{k+1}}$ can be chosen. This implies that $ |B_n(A)| = 2^{\frac{n}{2^k}}$ for every $n \geq 2^k$, which in turn implies that $$h_{top}(\Sigma_{\omega_{a^*_k}^{\prime \prime \prime}}) = \dfrac{1}{2^k}.$$ 
\end{proof}

In particular $h_{top}(\Sigma_a) = \frac{1}{2}$ if $a = 11(01)^{\infty}$. Observe that from Lemma \ref{powersoftwo}, given $i \in \mathbb{N}$ we find that a particular infinite sequences satisfy that $h_{top}(\Sigma_a) = \frac{1}{2^i}$. Then it is natural to ask if there is a finite $i$-sequence $\omega$ such that $\Sigma_{\omega}$ satisfies the same property. However, we will show that this is not possible. This fact is a consequence of Proposition \ref{exceptional1} and Theorem \ref{plateaus2} below. The following Lemma gives an estimate for the length of $i$-irreducible words.
 
\newtheorem{entropyintervals}[subsection]{Lemma}
\begin{entropyintervals}
Let $\omega$ be an $i$-irreducible sequence. If $$\dfrac{1}{2^{i+1}} < h_{top}(\Sigma_{\omega}) < \dfrac{1}{2^{i}},$$ then $\ell(\omega) \geq 3\cdot2^{i}$. \label{entropyintervals}
\end{entropyintervals}
\begin{proof}
Let $i \geq 0$. By Lemma \ref{powersoftwo} we obtain that $\omega_{a^*_{i+1}}^{\prime \prime \prime} \prec \omega \prec \omega_{a^*_i}^{\prime \prime \prime}$. Consider $a = \omega_{a^*_{i+1}}(\bar \omega_{a^*_{i+1}})^2$. Note that $\omega_{a^*_{i+1}}^{\prime \prime \prime} \prec a \prec \omega_{a^*_i}^{\prime \prime \prime},$ and $\ell(a) = 3\cdot2^i$. 

Suppose that there exist a finite word $\omega$ such that $\omega_{a^*_{i+1}}^{\prime \prime \prime} \prec \omega \prec \omega_{a^*_i}^{\prime \prime \prime},$ and $\ell(a) > \ell(\omega)$. Note that the first $3\cdot2^i-1$ symbols of $\omega_{a^*_{i+1}}^{\prime \prime \prime}$ and $\omega_{a^*_i}^{\prime \prime \prime}$ coincide, therefore, $w_1 \ldots w_{2^{i+1}} = a_1 \ldots a_{i+1}$. This implies that there exist $j \in \{2^{i+1} \ldots 3\cdot2^i\}$ such that $(\omega_{a^*_{n-1}}^{\prime \prime \prime})_j = 0 $ and $w_j = 1$, therefore, $\omega \succ \omega_{a^*_i}^{\prime \prime \prime}$, which is a contradiction. 
\end{proof}

\newtheorem{irreducibles1}[subsection]{Lemma}
\begin{irreducibles1}
If $\omega, \upsilon$ are finite $i$-irreducible sequences and $\omega \neq \upsilon$, then $(\omega, \omega^{\prime \prime \prime}) \cap (\upsilon, \upsilon^{\prime \prime \prime}) = \emptyset.$ \label{irreducibles1}
\end{irreducibles1}
\begin{proof}
Let $\omega = w_1 \ldots w_{\ell(\omega)}$ and $\upsilon = v_1 \ldots v_{\ell(\upsilon)}$ be $i$-irreducible words with $i \geq 0$. Without generality suppose that $\omega \prec \upsilon$. It is already assumed $(\omega, \omega^{\prime \prime \prime}) \cap (\upsilon, \upsilon^{\prime \prime \prime}) \neq \emptyset.$ Then, by Proposition \ref{interval1}, $(\upsilon, \upsilon^{\prime \prime \prime}) \subset (\omega, \omega^{\prime \prime \prime}).$ This implies that $\ell(\upsilon) \geq \ell(\omega)$ and that $v_j = w_j$ for every $j \in \{1, \ldots , \ell(\omega) \}$. By the $i$-irreducibility of $\upsilon$, $\omega^{\prime \prime \prime} \prec \upsilon$, which is a contradiction.
\end{proof}

\newtheorem{irreducibles2}[subsection]{Lemma}
\begin{irreducibles2}
For every finite sequence $\upsilon \succ a^*$ there exist an unique $i$-irreducible sequence $\omega$ such that $\upsilon \in [\omega, \omega^{\prime \prime \prime}]$. \label{irreducibles2}
\end{irreducibles2}
\begin{proof}
Since Proposition \ref{powersoftwo} gives us a partition of the lexicographic interval $[a^*, 1^{\infty}]$ in terms of the entropy function,  $\upsilon$ is an $i$-sequence for some $i \geq 0$. If $\upsilon$ is an $i$-irreducible sequence then $\omega = \upsilon$ satisfies the conclusion. Let $\upsilon$ be an $i$-sequence for some $i \geq 0$ such that it is not $i$-irreducible. Then there exists $2 < j < \ell(\upsilon)$ such that $\upsilon_j^{\prime \prime \prime} \geq \upsilon$, where $\upsilon_j = v_1 \ldots v_j$ and $v_j = 1$. By Lemma \ref{interval1}, $\upsilon \in (\upsilon_j, \upsilon_j^{\prime \prime \prime})$. Therefore, if $\upsilon_j$ is $i$-irreducible then $\omega = \upsilon_j$. Suppose that for every $2 < j < \ell(\upsilon)$, $\upsilon_j$ is not $i$-irreducible. Note that $v_{2^{i+1}}=1$ because if $v_{2^{i+1}} = 0$ then $\upsilon$ will not be an $i$ sequence. Indeed $v_1 \ldots v_{2^{i+1}}= \omega_{a^*_{i+1}}$. This implies that $\upsilon \prec (v_1\ldots v_{2^{i+1}})^{\prime \prime \prime}$, which contradicts that $\upsilon$ is an $i$-sequence. The uniqueness of such $i$-irreducible sequence $\omega$ is given by Lemma \ref{irreducibles1}.
\end{proof}

The following results show that the condition for being an entropy plateau is satisfied by intervals of the form $[\omega, \omega^{\prime \prime \prime}]$. For this purpose, we show that every symmetric sub-shift has a unique transitive component of maximal entropy. Given a dynamical system $(X,f)$, a subset $A$ of $X$ is a \textit{transitive component} if $A$ is closed, completely invariant (i.e. $f^{-1}(A) = A = f(A)$), $f\mid_A: A \to A$ is topologically transitive and there is no other set $A^{\prime}$ such that $A \subsetneq A^{\prime}$ containing a dense orbit \cite[p. 1313]{bundfuss}. In \cite[Theorem 6.3]{bundfuss} it is showed that every interval exclusion system has at most $4k$ transitive components, where $k$ is the number of holes (considering each hole as a connected subset). They also show in \cite[Theorem 6.4]{bundfuss} that if the exclusion sub-shift is a shift of finite type then it has at most $2k$ components. Observe that we can consider the family of symmetric sub-shifts as exclusion sub-shifts of the doubling map (see Section \ref{exclusion}). Moreover, by \cite[Theorem 6.4]{bundfuss} we assure that there exists a subset $A$ of $\Sigma_{\omega}$ such that $A$ is a transitive component of maximal entropy. 

\newtheorem{controlentropy}[subsection]{Lemma}
\begin{controlentropy}
Let $a \in[a^*, 1^{\infty}]$ and $\omega$ an $i$-irreducible finite sequence for $i \in \mathbb{N}$. If $a \in (\omega, \omega^{\prime \prime \prime}]$, then $h_{top}(A) \leq \frac{1}{\ell(\omega)} < h_{top}(\Sigma_{\omega})$, where $A$ is any transitive component of $\Sigma_a$ such that $h_{top}(A) < h_{top}(\Sigma_a)$.\label{controlentropy}
\end{controlentropy}
\begin{proof}
Let $\omega = \omega_1 \ldots \omega_{\ell(\omega)}$ be an $i$-irreducible sequence and $a \in (\omega, \omega^{\prime \prime \prime}]$. By Proposition \ref{consecutive} $a \in (1^n, 1^{n+1}]$ for some $n \geq 2$. Note that $\Sigma_{\omega} \subset \Sigma_a$ and $$(\Sigma_a \setminus \Sigma_{\omega}) = \{x \in \Sigma_a \mid \omega \hbox{\rm{ or }} \bar \omega \hbox{\rm{ occurs in }} x\}.$$ Consider $A \subset (\Sigma_a \setminus \Sigma_{\omega})$, a transitive component given by \cite[Theorem 6.3]{bundfuss}. Since $A$ is a transitive component then $A$ must be $\sigma_a$-invariant. Therefore, it is only required to consider words such that $\omega$ or $\bar \omega$ occurs in the first $\ell(\omega)$ positions. Furthermore, $|B_n(A)| \geq 0$ if $n \geq \ell(\omega)$, otherwise is zero. Moreover, $|B_n(A)| \geq 2$ for $n > \ell(\omega)$ and $|B_{\ell(\omega)}(A)| = 2$. Then, for every $1 \leq i < \omega$, if $u \in B_{\ell(\omega)+i}$ the $\ell(\omega)+j$ entry is fixed for every $j \leq i$ because $a \preccurlyeq \omega^{\prime\prime\prime}$. Indeed, if $u$ starts with $\omega$ then $u_{\ell(\omega)+i} = \bar w_i$. Hence $|B_{\ell(\omega)+i}| \leq 2^{\frac{n}{\ell(\omega)}}$. Notice that the digit $u_{2\ell(\omega)}$ is can be chosen from $0$ or $1$. Then $|B_{2\ell(\omega)}(A)| = 4$.  Therefore, $|B_n(A)| \leq 2^{\frac{n}{\ell (\omega)}}$, which implies that $$h_{top}(A) =  \mathop{\lim}\limits_{n \to \infty} \dfrac{1}{n} \log |B_n(A)| \leq \dfrac{1}{\ell(\omega)}.$$ Observe that by Lemmas \ref{powersoftwo} and \ref{entropyintervals} we can concluded that $h_{top}(A) < h_{top}(\Sigma_{\omega}).$
\end{proof}

\newtheorem{component}[subsection]{Theorem}
\begin{component}
Let $i \in \mathbb{N}$. Then, for every $i$-irreducible sequence $\omega$ and $a \in [\omega, \omega^{\prime \prime \prime}]$, $\Sigma_{a}$ contains a unique transitive component $A$ such that $h_{top}(A) = h_{top}(\Sigma_{\omega})$.\label{component}
\end{component}
\begin{proof}
Since $\omega$ is $i$-irreducible and $i \in \mathbb{N}$ then $\omega = 11(01)^k001\ldots$ for some $k \geq 1$. By Theorem \ref{critical}, the blocks $u= [(10)^{k-1}11]$, $v=[(10)^k]$ and their mirror images are admissible blocks of length $2^k$. Furthermore, there is no $w \in \mathcal{L}(\Sigma_{\omega})$ such that $uwv \in \mathcal{L}(\Sigma_{\omega})$. Define $$X = \{x \in \Sigma_{\omega} \mid u, v, \bar u \hbox{\rm{ or }} \bar v \hbox{\rm{ occurs in }} x\}.$$ Note that if $u$ or $\bar u$ occur in $x$ then there exist $N \geq \ell(\omega)$ such that $\sigma^n(x) \in \Sigma_{\omega} \setminus X$ for every $n \geq N$.  Besides if $v,\bar v$ occur in $x$, then $x$ has to be $(01)^{\infty}$ or $(10)^{\infty}$. Therefore $\{(01)^{\infty}, (10)^{\infty}\}$ is a transitive component with zero topological entropy, which we denote by $A_1$. By construction, $A_2$ is a sub-shift of finite type. Then, by \cite[Proposition 2.5.5]{brin}, $$h_{top}(\Sigma_{\omega}) = \max \{h_{top}(A_2), h_{top}(A_1)\} = h_{top}(A_2).$$
 
Suppose that there exists another transitive component $A$ such that $A_2 \cap A = \emptyset$ and $h_{top}(A) = h_{top}(\Sigma_{\omega}).$ Then, this component has to be different from $\{(01)^{\infty}, (10)^{\infty}\}$, and $A \subset X$. Since $A \neq \{(01)^{\infty}, (10)^{\infty}\}$ then for every $x \in A$, $u$ or $\bar u$ occur in $x$, therefore $A$ is not $\sigma$ invariant, then $A$ is not a transitive component.
\end{proof}

The unique transitive component $A$ of a non-transitive sub-shift given by Theorem \ref{component} will be called \textit{the main component of $\Sigma_{\omega}$}. As a consequence of Lemmas \ref{irreducibles2} and \ref{controlentropy} the conclusion of Theorem \ref{component} also hold for $\omega \succ 11(01)^{\infty}$. In such a case, the main component of $\Sigma_{\omega}$ is parametrised by an irreducible word, and by Theorem \ref{transitivity} $(\Sigma_{\omega}, \sigma_{\omega})$ is a transitive sub-shift of finite type. 

\newtheorem{controlentropy2}[subsection]{Lemma}
\begin{controlentropy2}
Let $a \in[a^*, 1^{\infty}]$ and $\omega$ an $i$-irreducible word for $i \geq 0$. If $a \notin (\omega, \omega^{\prime \prime \prime})$, then $h_{top}(\Sigma_a) \neq h_{top}(\Sigma_{\omega})$. \label{controlentropy2}
\end{controlentropy2}
\begin{proof}
Without generality suppose that $\omega^{\prime \prime \prime} \prec a$. Then, there exists $k \in \mathbb{N}$ such that $\omega^{\prime\prime\prime}_k = 0$ and $a_k = 1$. Let $\nu = a_1 \ldots a_k$. Recall that $$h_{top}(\Sigma_{\omega}) \leq h_{top}(\Sigma_{\nu}) \leq h_{top}(\Sigma_a),$$ since $\Sigma_{\omega^{\prime \prime \prime}} \subset \Sigma_{\nu} \subset \Sigma_a$. 

Note that $\nu \notin (\omega, \omega^{\prime \prime \prime})$. Let $A_{\omega}$ and $A_{\nu}$ be the main components of $\Sigma_{\omega}$ and $\Sigma_{\nu}$ respectively. Observe that $\nu_{\min}$ and $\nu_{\max} \in \mathcal{L}(A_{\nu})$. This implies for any $u, v \in \mathcal{L}(A_{\omega})$, there exist $w,z \in \mathcal{L}(\Sigma_{\nu})$ such that $uw[a_1\ldots a_{k-1}0]^jzv \in \mathcal{L}(A_{\nu})$ for any $j \in \mathbb{N}$. Moreover, $$uw[a_1\ldots a_{k-1}0]^jzv \in \mathcal{L}(A_{\nu} \setminus A_{\omega}).$$
Furthermore, $(a_1\ldots a_{k-1} 0)^{\infty}$ and $\overline{(a_1\ldots a_{k-1} 0)^{\infty}} \in A_{\nu} \setminus A_{\omega}$. Then, by Theorem \ref{component} and \cite[Corollary 4.4.9]{lindmarcus} we conclude $\Sigma_a > \Sigma_{\omega}$.
\end{proof}

\subsubsection*{Approximation properties and the exceptional set}

Two different ways to approximate sub-shifts in terms of shifts of finite type are given by the following definition.

\newtheorem{aproximate}[subsection]{Definition}
\begin{aproximate}
\normalfont{Let $(\Sigma, \sigma)$ be a sub-shift. We say that $(\Sigma, \sigma)$ is \textit{approximated from below} if there exist a sequence of sub-shifts $\{(\Sigma^n, \sigma_n)\}$ such that for every $n \in \mathbb{N}$:
\begin{itemize}
 \item [$i)$] $\Sigma^n \subset \Sigma^{n+1}$; 
 \item [$ii)$]$\Sigma = \overline{\mathop \bigcup \limits_{n=1}^{\infty} \Sigma^n}$;
 \item [$iii)$] there exist a homeomorphic copy $X^n$ of $\Sigma^{n}$ contained in $\Sigma^{n+1}$ such that $\sigma_{n+1} \mid_{X^n} = \sigma_n$;
 \item [$iv)$] there exist a homeomorphic copy $Y^n$ of $\Sigma^{n}$ contained in $\Sigma$ such that $\sigma \mid_{Y^n} = \sigma_n$.
\end{itemize}
Furthermore, we say that $(\Sigma, \sigma)$ is \textit{approximated from above}, if there exist a sequence of sub-shifts $\{(\Sigma^n, \sigma_n)\}$ such that for every $n \in \mathbb{N}$:
\begin{itemize}
 \item [$i^{\prime})$] $\Sigma^n \supset \Sigma^{n+1}$ and $\sigma_n \mid_{\Sigma^{n+1}} = \sigma_{n+1}$;
 \item [$ii^{\prime})$] $\Sigma = \mathop \bigcap \limits_{n=1}^{\infty} \Sigma^n$ and;
 \item [$iii^{\prime})$] $\sigma = \sigma_{n}\mid_{\Sigma}$.
 \end{itemize}
} \label{aproximate}
\end{aproximate}

Observe that we consider homeomorphic copies of $\Sigma^n$, $X^n$ and $Y^n$ for the approximation from below since we can not assure that $\Sigma^n$ is an invariant set neither for $\Sigma^{n+1}$ nor $\Sigma$. 

\newtheorem{aproximation}[subsection]{Theorem}
\begin{aproximation}
For any $a \in [1(0)^{\infty},1^{\infty})$, there exist sequences $\{a^+_n\}_{n=1}^{\infty}$ and $\{a^-_n\}_{n=1}^{\infty}$ of finite sequences such that: 
\begin{enumerate}
\item $\Sigma_a$ is approximated from below by $(\Sigma^{a^-_n}, \sigma_n)$;
\item $\Sigma_a$ is approximated from above by $(\Sigma^{a^+_n},\sigma_n)$.
\end{enumerate}\label{aproximation}
\end{aproximation}
\begin{proof}
Note that if $a$ is a finite sequence it suffices to consider the sequence $a_n = a$ and the sequence of sub-shifts $(\Sigma^n, \sigma_n) = (\Sigma_a, \sigma_a)$ for every $n \in \mathbb{N}$ to approximate from below and from above $(\Sigma_a, \sigma_a)$.  

Firstly, we will prove item \textit{(1)}. Consider an infinite sequence $a$. By Theorem \ref{teo1}, we may assume that $a_1 = a_2 = 1$, and by Proposition \ref{consecutive} we can assume that $a_1 = \ldots = a_n = 1$ for some $n \geq 2$. Consider $a^-_1 = 1^n$. Let $k_2 > n$ be the first index such that $a_k =1$ and $a_j = 0$ for every $n < j < k$. Such $k$ exists because $a$ is an infinite sequence. Let $a^-_2 = a_1 \ldots a_{k_2}$. Define $k_3$ as the first $k > k_2$ such that $a_k =1$ and $a^-_3 = a_1 \ldots a_{k_3}$. Then we can define a sequence $k_n$ to be the first index $k$ such that $k > k_{n-1}$ and $a_k =1$, and $a^-_n$ to be $a_1 \ldots a_{k_n}$. Note that $a^-_n \underset{n \to \infty}\longrightarrow a$. 

Note that $a^-_n$ is an increasing sequence such that $a^-_n \prec a^-_{n+1}$ for every $n \in \mathbb{N}$. Therefore, $\Sigma^{a^-_n} \subset \Sigma^{a^-_{n+1}}$ for every $n \in \mathbb{N}$, which gives $i)$ of Definition \ref{aproximate}. 

Note that Definition \ref{symmetric} and the fact that $\Sigma_a$ is a closed set imply $$\overline{\mathop \bigcup \limits_{n=1}^{\infty} \Sigma^{a^-_n}} \subset \Sigma_a.$$ Consider a cylinder $[\omega] \subset \Sigma_a$. Recall that there exist a one to one correspondence between cylinders of length $n$ and admissible words of length $n$. Therefore, we can consider $[\omega]$ as an admissible word of length $\ell(\omega)$. Let $n \in \mathbb{N}$ such that $\ell(a^-_n) \leq \ell(\omega) < \ell(a^-_{n+1})$. Then $\omega \in B_{\ell(\omega)}(\Sigma^{a^-_{n+1}})$. Therefore, there exist a point $x \in \Sigma^{a^-_{n+1}}$ such that $x \in [\omega]$, hence $\mathop \bigcup \limits_{n=1}^{\infty} \Sigma^{a^-_n}$ is dense on $\Sigma_a$, which gives us $ii)$.

Let $n \in \mathbb{N}$ and consider $a_n$ and $a_{n+1}$ and their associated sub-shifts $\Sigma^{a^-_n}$ and $\Sigma^{a^-_{n+1}}$. Let $$X^n = \{x \in \Sigma^{a^-_{n+1}} \mid a^-_n \hbox{\rm{ or }} \bar a^-_n \hbox{\rm{ does not occur in }}x\}.$$ Note that this set is $\sigma_{n+1}$ invariant and it is in bijective correspondence with $\Sigma^{a^-_n}$ by associating to each sequence $x \in \Sigma^{a^-_n}$ the same sequence $x$ in $X^n$. This gives us item $iii)$ of Definition \ref{aproximate}. Using the same argument we can construct $Y^n$ in $\Sigma_a$, satisfying the properties required by Definition \ref{aproximate} $iv)$. 

\vspace{1em}To show \textit{(2)}, note that the construction of the sequence is similar to item \textit{(1)} but we have to take certain considerations. Consider $a^-_n$ defined above. Let $a^+_1 = a^-_{k_j}1$, where $k_n$ is the first index such that $a_{k_j+1} = 0$ and $k_j > n$, where $n$ is given by $a^-_1$. Then let $a^+_i = a^-_{k_i}1$, where $i$ satisfies that $k_i$ is the first index such that $a_{k_i+1} = 0$ and $k_i > k_{i-1}$. Note that $a^+_n \underset{n \to \infty}\longrightarrow a$.  Observe that $a^+_n$ is a decreasing sequence such that $a^+_{n+1} \prec a^-_n$ for every $n \in \mathbb{N}$, then $\Sigma^{a^+_{n+1}} \subset \Sigma^{a^+_n}$ for every $n \in \mathbb{N}$. Given $x \in \Sigma^{a^+_{n+1}}$, $\sigma_n(x) \in \Sigma^{a^+_{n+1}}$, this shows that $\sigma_n(x) = \sigma_{n+1}(x)$, proving $i^{\prime})$. 

By construction $a \prec a^+_n$ for every $n \in \mathbb{N}$, then $\Sigma_a \subset \Sigma^{a^+_n}$. This fact implies that $\Sigma_a \subset \mathop{\bigcap}\limits_{n=1}^{\infty} \Sigma^{a^+_n}$. Let $x \in \mathop{\bigcap}\limits_{n=1}^{\infty} \Sigma^{a^+_n}$. Then $\bar a^+_n \prec \sigma_{a^+_n}^j(x) \prec a^+_n$ for every $j,n \geq 0$. Therefore, $\bar a \preccurlyeq \sigma^j(x) \preccurlyeq a$ for every $j \geq 0$. 

Note that for every $\sigma_a(x) = i_n(\sigma_a(x)) = \sigma_n(i_n(x))$, where $i_n: \Sigma_a \to \Sigma^{a^+_n}$ is the $n$th inclusion map. This proves item $iii^{\prime})$.
\end{proof}

We will call the approximations $(1)$ and $(2)$ given by Theorem \ref{aproximation} \textit{canonical approximations from below and above}, respectively.

\newtheorem{stabilityproperty}[subsection]{Theorem}
\begin{stabilityproperty}
Let $a \in \left[(10)^{\infty},1^{\infty}\right]$ be an infinite sequence. Then for every $m \in \mathbb{N}$ there exists $N \in \mathbb{N}$ such that $B_m(\Sigma_a) = B_m(\Sigma_{a_n}) = B_m(\Sigma_{a_N})$ for every $n \geq N$, where $a_n$ is a sequence such that $a_n \preccurlyeq a_{n+1}$ and $a_n \underset{n \to \infty}\longrightarrow a$. \label{stabilityproperty}
\end{stabilityproperty}
\begin{proof}
Let $a$ be an infinite sequence and $a_n$ be the sequence given by Theorem \ref{aproximation}(1). Note that this sequence satisfies $a_n < a_{n+1}$ and $a_n \underset{n \to \infty}\longrightarrow a$. Recall that $a_i = 1$ for every $1 \leq i \leq k$ for some $k \in \mathbb{N}$. If $m < k$ then $B_m(\Sigma_{a_n})$ is determined by $B_m(X_{\mathcal{F}})$, where $\mathcal{F} = \{0^k, 1^k\}$ and $B_m(\Sigma_{a_n}) = B_m(\Sigma_a)$ for every $n \in \mathbb{N}$. Besides, if $m = k$ then $$B_m(\Sigma_a) = B_m(\Sigma_{a_n}) = B_m(X_{\mathcal{F}^{\prime}}),$$ where $\mathcal{F}^{\prime} = \{0^{k+1}, 1^{k+1}\}$. Suppose that $m > k$. Let $N \in \mathbb{N}$ such that $a_N$ is an element of $\{a_n\}_{n=1}^{\infty}$ such that $\ell(a_{N-1}) \leq m < \ell(a_N)$. It suffices to show that $B_m(\Sigma_{a_N}) = B_m(\Sigma_{a_{N+1}})$. Suppose that $B_m(\Sigma_{a_N}) \neq B_m(\Sigma_{a_{N+1}})$, then by Definition \ref{symmetric} $|B_m(\Sigma_{a_N})| < |B_m(\Sigma_{a_{N+1}})|$. This implies that there exist $\omega \in B_m(\Sigma_{a_{N+1}})$ such that $\omega \succ u$ and $\bar \omega \prec u$ for every $u \in B_m(\Sigma_{a_N})$. Let $v = v_1 \ldots v_m$ be the maximal element of $B_m(\Sigma_{a_N})$. This implies that there exists $k \leq m$ such that $w_k = 1$ and $v_k =0$. Note that the first $m$ terms of the word $a_{N+1}$ are equal to $v_i$. Then $u \succcurlyeq a_{N+1}$, which is a contradiction.  
\end{proof}

\newtheorem{stabilityproperty1}[subsection]{Theorem}
\begin{stabilityproperty1}
Let $a \in \left[(10)^{\infty},1^{\infty}\right]$ be an infinite sequence. Then for every $m \in \mathbb{N}$ there exist $N \in \mathbb{N}$ such that $B_m(\Sigma_a) = B_m(\Sigma_{a_n})$ for every $n \geq N$, where $a_n$ is a sequence such that $a_n \succcurlyeq a_{n+1}$ and $a_n \underset{n \to \infty} \longrightarrow a$. \label{stabilityproperty1}
\end{stabilityproperty1}
\begin{proof}
Let $a$ be an infinite word and $a_n$ be the sequence given by item $(2)$ of Theorem \ref{aproximation}. Let $m \in \mathbb{N}$. Let $N \in \mathbb{N}$ such that $\ell(a_{N-1}) < m \leq \ell(a_N)$. Note that $B_m(\Sigma_a)$ is determined by the first $m$ symbols of $a$, also, $a_n$ coincides with $a$ in the first $\ell(a_N)$ symbols for every $n \geq N$, then $B_m(\Sigma_{a_n})=B_m(\Sigma_a)$ for every $n \geq N$.
\end{proof}

\newtheorem{exceptional1}[subsection]{Theorem}
\begin{exceptional1}
A sequence $a \in \mathcal{E}$ if and only if $a$ is an infinite sequence such that it can be approximated from below by sub-shifts of finite type given by $i$-irreducible words, i.e, there exists a sequence $\{a^-_n \}$ of $i$-irreducible sequences such that $(\Sigma^{a^-_n}, \sigma_{a^-_n})$ is a sub-shift of finite type for every $n$ and $(\Sigma_a, \sigma_a)$ is approximated from below by $(\Sigma^{a^-_n}, \sigma_{a^-_n})$ .\label{exceptional1}
\end{exceptional1}
\begin{proof}
Let $a$ be an infinite sequence approximated from below by $\{a^-_n\}_{n=1}^{\infty}$, where $a^-_n$ is an $i$-irreducible word. Suppose that there exist an $i$-irreducible sequence $\omega$ such that $a \in (\omega, \omega^{\prime \prime \prime})$. Let $N \in \mathbb{N}$ be sufficiently large such that $\ell(\omega) < \ell(a^-_N)$. Note that $\omega \prec a^-_N \prec a$. This implies that $a^-_N \in (\omega, \omega^{\prime \prime \prime})$, which contradicts Lemma \ref{irreducibles1}. Therefore, $a \in \mathcal{E}$.

Suppose that $a \in \mathcal{E}$. Note that $a$ is an $i$-sequence for some $i \geq 0$. Let $n_i$ be a sequence of positive integers such that $n_i < n_{j+1}$ and $a_{n_j} = 1$ for every $j \in \mathbb{N}$. From Lemma \ref{irreducibles1}, there exist $a^-_j$ such that $a_1 \ldots a_{n_j} \in (a^-_j,{a^-_j}^{\prime \prime \prime})$. The sequence $a^-_j$ satisfies Definition \ref{aproximate}.
\end{proof}

Summing up from Proposition \ref{interval1} to Lemma \ref{exceptional1} we obtain the following theorem.

\newtheorem{plateaus2}[subsection]{Theorem}
\begin{plateaus2}
Every entropy plateau is of the form $\left[\omega, \omega^{\prime \prime \prime}\right]$, where $\omega$ is an $i$-irreducible finite sequence. \label{plateaus2}
\end{plateaus2}

To conclude that the entropy function is a devil's staircase, note that $\mathcal{E} \subset P$ where $P$ is the set of symmetric Parry sequences defined in Section \ref{exclusion} which has Lebesgue measure zero.

\section{Specification} 
\label{specificationsection}

\noindent A sub-shift $(\Sigma, \sigma)$ has \textit{the specification property (or simply $(\Sigma, \sigma)$ has specification)} if there exist $m \in \mathbb{N}$ such that for any $u,v \in \mathcal{L}(\Sigma)$ there exist $w \in B_m(\Sigma)$ such that $uwv \in \mathcal{L}(\Sigma)$, i.e. every two words $u$ and $v$ can be connected by a word $w$ of length $m$ \cite[p.66]{boyle}. Note that every sub-shift with specification is transitive. Alternatively, we can state the specification property in the following way: Let $(\Sigma, \sigma)$ be a sub-shift. Define 
\begin{align*}
m_n = \inf \{k \mid &\hbox{\rm{ for every }} u, v \in B_n(\Sigma) \hbox{\rm{ there exist }} w \in B_k(\Sigma)\\ 
&\hbox{\rm{ such that }} uwv \in \mathcal{L}(\Sigma) \}.
\end{align*}

Then $(\Sigma, \sigma)$ has specification if and only if $\mathop{\lim}\limits_{n \to \infty} m_n < \infty$ (the limit exists since $m_{n+1} \leq m_n$). We call to $\mathop{\lim}\limits_{n \to \infty} m_n < \infty$ \textit{the specification number of $(\Sigma, \sigma)$}. We want to characterise symmetric sub-shifts with specification using properties of their defining words. Recall that if $\omega = 1^n$ for some $n \geq 2$ then any pair of words $u, v \in \mathcal{L}(\Sigma_{\omega})$ can be connected by $0^{n-1}$. Also, in \cite{parry}, it was shown that every transitive sub-shift of finite type has specification. Then, it is natural to consider the specification number of a symmetric sub-shift $(\Sigma_{\omega}, \sigma_{\omega})$ given by an irreducible word $\omega$. In this case the specification number will be denoted by $s_{\omega}$. Using the canonical approximations given by Theorem \ref{aproximation} and the specification number of each element of the approximation from above we prove the main theorem of this section.

\newtheorem{specification}[subsection]{Theorem}
\begin{specification}
Let $a \in \mathcal{E}\cap [11(01)^{\infty}, 1^{\infty})$ and $a \in (1^n, 1^{n+1})$ for some $n \geq 2$. Then
\begin{enumerate}
\item If $0^n$ does not occur in $a$ then $(\Sigma_a, \sigma_a)$ has specification;
\item If $0^n$ occurs finite times, then $(\Sigma_a, \sigma_a)$ has specification;
\item If $0^n$ occurs infinitely many times, then $(\Sigma_a, \sigma_a)$ has specification, if $a_k^-$ satisfies Lemma \ref{splemma7} for every $k \in \mathbb{N}$.
\end{enumerate}
\label{specification}
\end{specification}

As it was mentioned in Section \ref{resultstransitivityandentropy} from the proof of Theorem \ref{transitivity} $s_{\omega} \leq 2\ell(\omega)$. Moreover, observe that for every $n \geq 3$ an irreducible word $\omega \in (1^n, 1^{n+1})$ has at least length $n+2$ and $5$ if $\omega \in (11(01)^{\infty}, 111)$. Also, as a consequence of Lemma \ref{irreducibles1}, there are no irreducible words of the form $1^n0^{n-1}10^n1\ldots$ for every $n \geq 3$, and by Theorem \ref{critical}, the same is true for $n = 2$. 

\vspace{1em}We say that a sub-shift $\Sigma$ is a \textit{coded system} if $$\Sigma = \overline{\mathop \bigcup \limits_{n=1}^{\infty} \Sigma^n},$$ where each $\Sigma^n$ is a transitive sub-shift of finite type. Note that coded systems are transitive sub-shifts. We say that a sub-shift $\Sigma$ is an \textit{almost sofic system} if for every $\varepsilon > 0$ there exist a sub-shift of finite type $\Sigma_{\varepsilon} \subset \Sigma$ such that $h_{top}(\Sigma) > h_{top}(\Sigma_{\varepsilon}) - \varepsilon$. It is worth pointing out now that from Definition \ref{aproximate}, and the continuity of the entropy function, it is clear that for any $a \in \left[(10)^{\infty}, 1^{\infty}\right] \cap \mathcal{E}$, $\Sigma_a$ is an almost sofic system. And if $a \in \mathcal{E}$ and $a \succ 11(01)^{\infty}$ then $\Sigma_a$ is a mixing coded system.

\vspace{1em}In the following Lemma, we give an upper bound of the specification number, which depends not on the length of the irreducible word $\omega$ but on the occurrence of strings of zeroes.

\newtheorem{splemma2}[subsection]{Lemma}
\begin{splemma2}
Let $\omega \in (1^n, 1^{n+1})$ be an irreducible word with $n \geq 3$. If the blocks $0^{n-1}$ or $1^{n-1}$ do not occur in $\omega$, then $s_{\omega} \leq 2(n-1)$.\label{splemma2}
\end{splemma2}
\begin{proof}
Let $\omega$ be an irreducible word satisfying the hypothesis. Without loosing generality it suffices to consider words $u$ and $v$ such that $u$ ends with $1$, $v$ starts with $1$ and $\ell(u) = \ell(v) = \ell(\omega)$. Note that for any $v$ starting with $1$, the word $0^{n-1}v$ is admissible, and if $u$ does contain neither the blocks $0^n$ nor $1^n$ then $0^{n-1}$ is a bridge between $u$ and $v$ of length $n-1$. Suppose that $0^n$ or $1^n$ occurs in $u$. Let $$j = \max \{k \in \{0, \ldots \ell(\omega)-2n \} \mid \sigma^k(u) = 0^n1 \ldots u_{\ell(\omega) - 1}1 \hbox{\rm{ or }} 1^n0 \ldots u_{\ell(\omega)-1} 1 \}.$$ If $\sigma^j(u) = 1^n0\ldots u_{\ell(\omega-1)}1$, then $0^{n-1}$ is a bridge between $u$ and $v$ of length $n-1$. If $\sigma^j(u) = 0^n1\ldots 1$, then $\sigma^j(u) \succ \omega_{\min_{\ell(\omega)-j}}$ or $\sigma^j(u) = \omega_{\min_{\ell(\omega)-j}}$. If $\sigma^j(u) \succ \omega_{\min_{\ell(\omega)-j}}$ then $0^{n-1}$ is a bridge between $u$ and $v$ of length $n-1$. If $\sigma^j(u) = \omega_{\min_{\ell(\omega)-j}}$, then there exists $1 \leq k \leq n-2$ such that $u1^k$ is admissible. Then $1^k0^{n-1}$ is a bridge between $u$ and $v$ with length at most $2n-2$.
\end{proof}

\newtheorem{specification1}[subsection]{Proposition}
\begin{specification1}
Let $a \in \mathcal{E}$ such that $a \in \left[1^n, 1^{n+1} \right]$ for some $n \geq 3$. If $0^{n-1}$ does not occur in $a$ then $(\Sigma_a, \sigma_a)$ has specification.\label{specification1}
\end{specification1}
\begin{proof}
Let $a$ be an infinite sequence satisfying the hypotheses of the Theorem. Let $a^+_r$ be given by Theorem \ref{aproximation} $(2)$. Observe that the block $0^{n-1}$ does not occur in $a^+_n$ for any $n \in \mathbb{N}$. By Theorem \ref{stabilityproperty1}, $m_r = s_{a^+_{N(r)}}$. By Lemma \ref{splemma2}, $m_r \leq 2(n-1)$ for every $n$, whence $$\mathop{\lim}\limits_{r \to \infty} m_r < 2(n-1).$$ Therefore, $\Sigma_a$ has specification.  
\end{proof}

The presence of the blocks $0^{n-1}, 1^{n-1}, 0^n$ and $1^n$ in an infinite sequence $a \in (1^n, 1^{n+1})$ will increase the specification number of each member of the canonical approximation given by Theorem \ref{aproximation} $(2)$, in certain situations. Firstly, we show that the occurrence of $0^{n-1}$ will have no mayor affect on the specification number of $\Sigma_{\omega}$ if $d(\omega, 1^n(0^{n-1}1)^{\infty})$ is bounded away from zero.

\newtheorem{splemma3}[subsection]{Lemma}
\begin{splemma3}
Let $\omega \in (1^n, 1^{n+1})$ be an irreducible word with $n \geq 3$. If $d(\omega, 1^n(0^{n-1}1)^{\infty}) \geq \frac{1}{2^{2n-1}}$ and $0^n$ or $1^n$ does not occur in $\omega$, then $s_{\omega} \leq 2n$.\label{splemma3}
\end{splemma3}
\begin{proof}
The argument of this proof is similar to the one used for Lemma \ref{splemma2}. Let $\omega$ be an irreducible word satisfying the hypothesis. Observe that $d(\omega, 1^n(0^{n-1}1)^{\infty}) \geq \frac{1}{2^{2n-1}}$, implying that the block $[(0^{n-1}1)]^k$ does not occur after $1^n$ for any $k \in \mathbb{N}$. Let $u$ and $v$ be admissible words such that $0^n$ or $1^n$ occurs in $u$ and $\ell(u) = \ell(v) = \ell(\omega)$. Let $$j = \max \{k \in \{0, \ldots \ell(\omega)-2n \} \mid \sigma^k(u) = 0^n1 \ldots u_{\ell(\omega) - 1}1 \hbox{\rm{ or }} 1^n0 \ldots u_{\ell(\omega)-1} 1 \}.$$ If $\sigma^j(u) = 1^n0\ldots u_{\ell(\omega-1)}1$, then $0^{n-1}$ is a bridge between $u$ and $v$ with length $n-1$. If $\sigma^j(u) = 0^n1\ldots 1$, then $\sigma^j(u) \succ \omega_{\min_{\ell(\omega)-j}}$ or $\sigma^j(u) = \omega_{\min_{\ell(\omega)-j}}$. If $\sigma^j(u) \succ \omega_{\min_{\ell(\omega)-j}}$ then $0^{n-1}$ is a bridge between $u$ and $v$ of length $n-1$. If $\sigma^j(u) = \omega_{\min_{\ell(\omega)-j}}$, then there exists $1 \leq k \leq n-2$ such that $u1^k$ is admissible. Then $1^k0^{n-1}$ is a bridge between $u$ and $v$ with length at most $2n$.
\end{proof}

Observe that the proof of Proposition \ref{specification1} also holds for infinite sequences that can be approximated from above by words satisfying the hypothesis of Lemma \ref{splemma3}. Then the following theorem is true. The proof is left to the reader.

\newtheorem{specification2}[subsection]{Theorem}
\begin{specification2}
Let $a \in \mathcal{E}$ such that $a \in (1^n, 1^{n+1})$ for some $n \geq 2$. If $d(\omega, 1^n(0^{n-1}1)^{\infty}) \geq \frac{1}{2^{2n-1}}$ and $0^n$ or $1^n$ does not occur in $\omega$ then $(\Sigma_a, \sigma_a)$ has specification.\label{specification2}
\end{specification2}

Given an irreducible word $\omega$, we show that the specification number of $\Sigma_{\omega}$ will depend on the distance between $\omega$ and $1^n(0^{n-1}1)^{\infty}$.

\newtheorem{splemma4}[subsection]{Lemma}
\begin{splemma4}
Let $\omega \in (1^n, 1^{n+1})$ be an irreducible word with $n \geq 2$. If $0 < d(\omega, 1^n(0^{n-1}1)^{\infty}) \leq \frac{1}{2^{n(k+1)+1}}$ for $k \in \mathbb{N}$ and $0^n$ or $1^n$ does not occur in $\omega$, then $nk < s_{\omega} < n(k+2)$.\label{splemma4}
\end{splemma4}
\begin{proof}
Let $k \in \mathbb{N}$ and $\omega$ be an irreducible word such that $$d(\omega, 1^n(0^{n-1}1)^{\infty}) \leq \frac{1}{2^{n(k+1)+1}}.$$ Then the first $n(k+1)$ symbols of $\omega$ coincide with  $1^n(0^{n-1}1)^k$.  Let $u$ and $v$ be admissible words such that $0^n$ or $1^n$ occurs in $u$ and $\ell(u) = \ell(v) = \ell(\omega)$. Let $$j = \max \{k \in \{0, \ldots \ell(\omega)-2n \} \mid \sigma^k(u) = 0^n1 \ldots u_{\ell(\omega) - 1}1 \hbox{\rm{ or }} 1^n0 \ldots u_{\ell(\omega)-1} 1\}.$$ By our choice of $j$, if $\sigma^j(u)$ starts with $0^n$, whence $\sigma^j(u) = 0^n(1^{n-1}0)^r1^i$, with $0 \leq r < k$ and $1 \leq i \leq n-1$. By symmetry, if $\sigma^j(u)$ starts with $1^n$ then $\sigma^j(u) = 1^n(0^{n-1}1)^r$,  with $0 \leq r < k$ and $1 \leq i \leq n-1$. Let $w$ be such that $uw$ is an admissible word. Then the first $n-1-i$ symbols of $w$ are forced to be $1$ if $\sigma^j(u)$ starts with $0^n$ and to be $0$ if $\sigma^j(u)$ starts with $1^n$. The symbol $n-i$ is free. Note that for any choice of the symbol $n-i$, the following $s(n-1)$ symbols will be forced for $1 \leq s \leq k$. Then $nk$ is a lower bound for $s_{\omega}$. To compute the upper bound note that the symbol $sn+l$ is free for some $2 \leq l \leq n-1$. Then $s_{\omega} \leq nk+l+n-1 \leq nk + 2(n-1) < n(k+2)$.
\end{proof}

Observe that Lemma \ref{splemma4} shows that the specification number depends only on the first $n(k+1)$ digits of $\omega$ if $$d(\omega, 1^n(0^{n-1}1)^{\infty}) \leq \dfrac{1}{2^{n(k+1)+1}}.$$ 

\newtheorem{specification3}[subsection]{Theorem}
\begin{specification3}
Let $a \in \mathcal{E} \cap (1^n, 1^{n+1})$ $n \geq 3$. If $$0 < d(a, 1^n(0^{n-1}1)^{\infty}) \leq \dfrac{1}{2^{n(k+1)+1}}$$ for $k \in \mathbb{N}$ and $0^n$ or $1^n$ do not occur in $\omega$, then $(\Sigma_a, \sigma_a)$ has specification.\label{specification3}
\end{specification3}
\begin{proof}
Let $a \in \mathcal{E} \cap (1^n, 1^{n+1})$. Then there exists $k \in \mathbb{N}$ such that $$\frac{1}{2^{n(k+2)+1}} < d(a, 1^n(0^{n-1}1)^{\infty}) \leq \frac{1}{2^{n(k+1)+1}}.$$ Let $a^+_r$ be the canonical approximation from above given by Theorem \ref{aproximation}. Then for every $r \in \mathbb{N}$, the first $kn$ symbols of $a^+_n$ are determined by $1^n(0^{n-1}1)^k$. This implies that $a^+_n$ satisfies the hypothesis of Lemma \ref{splemma4} for every $r$, i.e. $$nk < s_{a^+_{N(r)}} < n(k+2).$$ By Theorem \ref{stabilityproperty1}, $m_r = s_{a^+_{N(r)}}$. Therefore, $$\mathop{\lim}\limits_{r \to \infty} m_r < n(k+2).$$ 
\end{proof}

It turns out that is needed to study irreducible words when the block $0^n$ (Lemmas \ref{splemma5}, \ref{splemma6} and \ref{splemma7}) in order to prove Theorem \ref{specification}. Notice that for every irreducible word $\omega$ such that $1^n \prec \omega \prec 1^{n+1}$ the block $0^n$ occurs in $\omega^{\prime}$ at least once. The situation explained in Lemma \ref{splemma4} is not limited only to irreducible words close to $1^n(0^{n-1}1)^{\infty}$.

\newtheorem{splemma5}[subsection]{Lemma}
\begin{splemma5}
Let $\upsilon$ be irreducible words such that $\upsilon \in (1^n, 1^{n+1})$ for a fixed $n \geq 2$. If there exists $\omega \in (1^n, 1^{n+1})$ such that $\omega$ does not contain the block $0^n$ and $$\dfrac{1}{2^{(k+1)\ell(\omega) + n}} \leq d(\omega^{\prime \prime \prime}, \upsilon) \leq \dfrac{1}{2^{(k+1) \ell(\omega)}}$$ for some $k \in \mathbb{N}$ and $0^n$ does not occur in $\nu$ after $(k+1)\ell(\omega)$ digits, then 
\begin{center}
$k\ell(\omega) \leq s_{\upsilon} \leq k\ell(\omega) + n(k+2).$ 
\end{center}
\label{splemma5}
\end{splemma5}
\begin{proof}
Let $\omega$ and $\upsilon$ be irreducible words satisfying the hypothesis. Then, $\upsilon$ contains the block $0^n$ at least $k$ times, and by Lemma \ref{irreducibles1} there is no $\ell(\omega)< j < (k+1)\ell(\omega)$ such that $\upsilon_1 \ldots \upsilon_j$ is irreducible, besides $\ell(\omega) < \ell(\upsilon)$. Indeed, $\ell(\upsilon) > (k+1)\ell(\omega)$. By Lemmas \ref{splemma2}, \ref{splemma3} and \ref{splemma4} we may confine ourselves to considering sequences which lie in $\Sigma_{\upsilon} \setminus \Sigma_{\omega}$ to compute bounds for $s_{\upsilon}$. 

Consider admissible words $u, v$ such that $\ell(u) = \ell(v) = \ell(\upsilon)$, $v$ starts with $1$ and $$u_{\ell(\upsilon) - (\ell(\omega)+1)} \ldots u_{\ell(u)} = \bar\omega 1.$$ Observe that if $z$ is a bridge between $u$ and $v$ then $z_1 = 1$. If $z_1 = 0$ then $\sigma^{\ell(\upsilon) - (\ell(\omega)+1)}0 \prec \bar \upsilon$. The same argument holds for $z_j$ for $j \leq n-1$. Observe that $z_n = 0$. If $z_n = 1$ then $uz$ will contain the block $1^{n+1}$ which is a contradiction. Therefore, $z$ will satisfy that $$\sigma^{\ell(\upsilon) - (\ell(\omega)+1)}(u) z = \bar\omega \omega_1 \ldots \omega_{\ell(\omega)-1},$$ then $z$ has at least length $\ell(\omega)-1$. Observe that the following digit is free. If $1$ is chosen then $z_1 \ldots z_{\ell(\omega)} = \omega$, which will force the following $\ell(\omega)-1$ digits to coincide with $\bar \omega$. If $0$ is chosen then $\sigma^{\ell(\upsilon) - (\ell(\omega)+1)}z = \bar \upsilon$, therefore, the next $\ell(\omega)-1$ digits of $z$ are forced by the previous argument. It has been shown that $z$ has a length of at least $k(\ell(\omega)) - 1$ and that after $\ell(\omega)-1$ steps we will have a choice. By Theorem \ref{transitivity}, the minimum number of steps to get from $\sigma^{\ell(\upsilon) - (\ell(\omega)+1)}(u)$ to $\upsilon_{\min_{(k+1)\ell(\omega)}}$ by $z$ is $\ell(\upsilon) - (\ell(\omega)+1)$. Substituting $\ell(\upsilon) > (k+1)\ell(\omega)$ in the minimum number of steps we obtain the desired lower bound for $s_{\upsilon}$. 

To obtain the upper bound, let $u$ be an admissible word with length $\ell(\upsilon)$ such that $u \notin B_{\ell(\upsilon)}(\Sigma_{\omega})$ such that $u$ ends with $1$. By hypothesis, $$\dfrac{1}{2^{(k+1)\ell(\omega) + n}} \leq d(\omega^{\prime \prime \prime}, \upsilon) \leq \dfrac{1}{2^{(k+1 \ell(\omega))}}.$$ Therefore, there exists an index $(k+1)\ell(\omega) < j < (k+1)\ell(\omega)+n$ such that $\upsilon_i  = 0$ for $\ell(\upsilon) \leq i < j$ and $\upsilon_j = 1$. This implies that the following $j-1$ symbols are forced to be $1$ and the symbol $j$ is free. If one is chosen then $z = 1^j$ with $0 \leq j \leq n-1$. By Lemma \ref{splemma4}, $z$ has at most $n(k+2)$ digits forced after its first $k\ell(\omega)$ digits. Therefore, the desired upper bound is obtained.
\end{proof}

Note that Lemma \ref{splemma5} deals with a particular family of sub-shifts where $0^n$ does not occur randomly in its defining word; namely, irreducible words such that they are close to a particular entropy plateau defined by a word without $0^n$ occurring in it. 

\newtheorem{spremark1}[subsection]{Proposition}
\begin{spremark1}
For every irreducible word $\upsilon$ such that $0^n$ occurs in $\upsilon$ there exists a word $\omega$ such that $\omega$ and $\upsilon$ satisfy the distance condition of Lemma \ref{splemma5} and $\omega$ does not contain $0^n$.  \label{spremark1}
\end{spremark1}
\begin{proof}
The result is just a consequence of Proposition \ref{interval1} and Lemma \ref{irreducibles1} applied to $\omega = 1^n$ and the given $\upsilon$.
\end{proof}

It is natural to ask if the specification number $s_{\upsilon}$ will change if $0^n$ occurs in $\upsilon$. As it will be shown, everything depends on the longest irreducible word, which does not contain $0^n$ occurring in $\upsilon$ and the way that $0^n$ is placed. As it is shown in the following generalisation of Lemma \ref{splemma5}, the specification number for an irreducible word depends entirely on the specification number of the irreducible subwords ocurring in $\upsilon$. The proof is essentially the same. 
 
\newtheorem{splemma6}[subsection]{Lemma}
\begin{splemma6}
Let $\upsilon$ be an irreducible word such that $\upsilon \in (1^n, 1^{n+1})$ for a fixed $n \geq 2$, and $0^n$ occurs in $\upsilon$. If there exist an irreducible word $\tau$ such that $$\dfrac{1}{2^{(k+1)\ell(\tau) + n}} \leq d(\tau^{\prime \prime \prime}, \upsilon) \leq \dfrac{1}{2^{(k+1) \ell(\tau)}}$$ for some $k \geq 0$, $\tau$ satisfies Lemma \ref{splemma5} and for any $j \leq |(\ell(\nu) -1)- k\ell(\tau)|$ such that $\upsilon_j = 1$ the word $\upsilon_1 \ldots \upsilon_j$ is not irreducible, then $$k\ell(\tau) \leq s_{\upsilon} \leq k\ell(\tau) + n(k+2).$$ \label{splemma6}
\end{splemma6}
\begin{proof}
Let $\upsilon$ be an irreducible word and suppose that there exists an irreducible word $\tau$ satisfying the hypothesis. By Lemmas \ref{splemma2}, \ref{splemma3}, \ref{splemma4} and \ref{splemma5} we only need to consider sequences that lie in $\Sigma_{\upsilon} \setminus \Sigma_{\tau}$ to compute bounds for $s_{\upsilon}$. 

Let $u, v$ be admissible words such that $\ell(u) = \ell(v) = \ell(\upsilon)$, $v$ starts with $1$ and $$u_{\ell(\upsilon) - (\ell(\tau)+1)} \ldots u_{\ell(u)} = \bar\tau 1.$$ Let $z$ be a bridge between $u$ and $v$. Applying the same construction as in Lemma \ref{splemma5} it is shown that $z$ has a length of at least $k(\ell(\tau)) - 1$. Applying Theorem \ref{transitivity}, the minimum number of steps to get from $\sigma^{\ell(\upsilon) - (\ell(\tau)+1)}(u)$ to $\upsilon_{\min_{(k+1)\ell(\tau)}}$ by $z$ is $\ell(\upsilon) - (\ell(\tau)+1)$. Substituting $\ell(\upsilon) > (k+1)\ell(\tau)$ in the minimum number of steps, we obtain the desired lower bound for $s_{\upsilon}$. 

To obtain the upper bound, let $u$ be an admissible word with length $\ell(\upsilon)$ such that $u \notin B_{\ell(\upsilon)}(\Sigma_{\tau})$ and $u$ ends with $1$. By hypothesis $$\dfrac{1}{2^{(k+1)\ell(\tau) + n}} \leq d(\omega^{\prime \prime \prime}, \upsilon) \leq \dfrac{1}{2^{(k+1 \ell(\tau))}}.$$ Therefore, there exists an index $(k+1)\ell(\tau) < j < (k+1)\ell(\tau)+n$ such that $\upsilon_i  = 0$ for $\ell(\upsilon) \leq i < j$ and $\upsilon_j = 1$. This implies that the following $j-1$ symbols are forced to be $1$ and the symbol $j$ is free. If one is chosen then $z = 1^j$ with $0 \leq j \leq n-1$. By Lemma \ref{splemma5}, $z$ has at most $n(k+2)$ digits forced after its first $k\ell(\tau)$ digits. Therefore, the desired upper bound is obtained. 
\end{proof}

\newtheorem{splemma7}[subsection]{Lemma}
\begin{splemma7}
Let $n \geq 2$ and $\upsilon \in (1^n, 1^{n+1})$ be an irreducible word such that for every irreducible subword $\omega^i$,  $0^n$ occurs in $\omega^i$, and $$\dfrac{1}{2^{2\ell(\omega^i)}} < d({\omega^i}^{\prime}, \upsilon) \leq \dfrac{1}{2^{\ell(\omega^i) + n}}.$$ Then $$s_{\upsilon} \leq \ell(\omega) + n,$$ where $\omega$ is the longest irreducible word such that $0^n$ does not occur in $\upsilon$.\label{splemma7}
\end{splemma7}
\begin{proof}
Let $\omega^1 = \omega$. By hypothesis $$\dfrac{1}{2^{2\ell(\omega^1)}} < d({\omega^1}^{\prime}, \upsilon) \leq \dfrac{1}{2^{\ell(\omega^1) + n}}.$$ This implies that there exist $\ell(\omega^1)+ (n+1) < j <  2\ell(\omega^1)$ such that ${\omega}^{\prime}_j = 0$ and $\upsilon_j = 1$. Moreover, $j \leq \ell(\upsilon)$ and $\tau = \upsilon_1 \ldots \upsilon_j$ is an irreducible word such that for any $j^{\prime} < j$ satisfying $\upsilon_{j^\prime} = 1$,  $\upsilon_1 \ldots \upsilon_{j^\prime}$ is irreducible. Firstly, we will calculate the specification number for $\Sigma_{\tau}$. 

Let $u \in B_{\ell(\tau)}(\Sigma_{\tau}) \setminus B_{\ell(\tau)}(\Sigma_{\omega})$ such that $u$ ends with $1$ and $v \in  B_{\ell(\tau)}(\Sigma_{\tau})$ such that $v$ starts with $1$. Then the blocks $[\omega]$ or $[\bar \omega]$ occur in $u$. Let $1 \leq k \leq \ell(\tau) - \ell(\omega)$ such that $\sigma^k(u)$ starts with $\omega$ or $\bar \omega$. Without loosing generality suppose that $\sigma^k(u)$ starts with $\bar \omega$. Observe that the word $z0^s$, where $z_i = \bar \tau_{\ell(\omega)+i}$ for $1 \leq i \leq (j-\ell(\omega))$ and $1 \leq s \leq n-1$, is a bridge between $u$ and $v$. Then $s_{\tau} \leq j-\ell(\omega^1) + n \leq \ell(\omega) + n$.  

Inductively, suppose that for every irreducible word, every irreducible subword $\omega^i$ is such that $$\dfrac{1}{2^{2\ell(\omega^i)}} < d({\omega^i}^{\prime}, \upsilon) \leq \dfrac{1}{2^{\ell(\omega^i) + n}},$$ then $$s_{\upsilon} \leq \ell(\omega) + n.$$ Let $\omega^{i+1}$ be the next irreducible word such that $$\dfrac{1}{2^{2\ell(\omega^{i+1})}} < d({\omega^{i+1}}^{\prime}, \upsilon) \leq \dfrac{1}{2^{\ell(\omega^{i+1}) + n}}.$$ Then, $\ell(\omega^2)+ (n+1) < \ell(\omega^3) <  2\ell(\omega^2)$, ${\omega^2}^{\prime}_j = 0$ and $\upsilon_{\ell(\omega^3)} = 1$. Observe that $0^n$ occurs in $\omega^{i+1}$ at least twice. Let $u \in B_{\ell(\omega^{i+1})}(\Sigma_{\omega^{i+1}}) \setminus B_{\ell(\omega^{i+1})}(\Sigma_{\omega^i})$ such that $u$ ends with $1$ and $v \in  B_{\ell(\omega^{i+1})}(\Sigma_{\omega^i})$ such that $v$ starts with $1$. Let $k \in \{0, \ldots, \ell(\omega^{i+1}) -\ell(\omega^i)\}$ such that $\sigma^k(u)$ starts with $\omega^i$ or $\bar \omega^i$. Without generality suppose that $\sigma^k(u)$ starts with $\bar \omega^i$. Observe that the word $z0^s$, where $z_i = \omega_i$ for $1 \leq i \leq (j-\ell(\omega^1))$ and $1 \leq s \leq n-1$, is a bridge between $u$ and $v$. Then $s_{\omega^{i+1}} \leq j-\ell(\omega^1) + n \leq \ell(\omega^1) + n$. 

Observe that there exists $I \leq \ell(\upsilon)$ such that $\omega^{I}$ satisfies the distance condition stated above and for any $i \geq I$, $0^n$ does not occur in $\omega^i$. Then $\ell(\omega)+n \leq s_{\upsilon}$. If $\ell(\upsilon) = I$ then $s_{\upsilon} = \ell(\omega) + n$. Suppose that $\ell(\upsilon) > I$. To show that $s_{\upsilon}$ is bounded from above, consider $u \in B_{\ell(\upsilon)}(\Sigma_{\upsilon}) \setminus B_{\ell(\upsilon)}(\Sigma_{I})$ and $v \in B_{\ell(\upsilon)}(\Sigma_{\upsilon})$ with $u$ ending with $1$ and $v$ ending with $1$. Then the block $(\bar \omega)^I$ occurs in $u$. Note that $\upsilon_{I+1} = 0$. If  $\upsilon_{I+1} = 1$ then $\upsilon_1 \ldots \upsilon_{I+1}$ is an irreducible word which contains $0^n$ and will satisfy the required distance condition, which is a contradiction to the choice of $I$. Let $t \leq \ell(\upsilon)$ such that $\upsilon_1 \ldots \upsilon_t$ is irreducible. Note that $t \leq I + n-1$. If $t \leq I + n$ then $0^n$ occurs in $\upsilon_t$ after $I$ symbols. Therefore there exist $1 \leq j \leq n-1$ such that the following $j$ symbols after $I+n$ symbols of $\bar \upsilon$ are forced to be $0$. Then $z = \bar \upsilon_{I+j} \ldots \bar \upsilon_{t-1}1 0^{n-1}$ is a bridge between $u$ and $v$ with $\ell(z) \leq 2n \leq \ell(\omega) + n$. Then $s_{\upsilon} = \ell(\omega)+n.$  
\end{proof}

\newtheorem{specification4}[subsection]{Theorem}
\begin{specification4}
If $a \in \mathcal{E}$ with $a \succ 11(01)^{\infty}$ and $a$ contains the block $0^n$ a finite number of times, then $\Sigma_a$ has specification. \label{specification4}
\end{specification4}
\begin{proof}
Let $a \in \mathcal{E}$ such that $0^n$ occurs $l$ times. Let $\{a^+_r\}_{r=1}^{\infty}$ the canonical approximation from above given by Theorem \ref{aproximation} $(2)$. By Theorem \ref{stabilityproperty1} $m_r = s_{a^+_{N(r)}}$. By hypothesis, there exist $N \in \mathbb{N}$ such that for every $j \geq N$ the block $0^n$ occurs $l$ times. Then: 
\begin{enumerate}[(i)]
\item For every $k \leq N$, $$\dfrac{1}{2^{2\ell(a^+_k)}} < d({a^+_k}^{\prime}, a^+_{k+1}) \leq \dfrac{1}{2^{\ell(a^+_N) + n}}.$$ In this case, by Lemma \ref{splemma7} for every $k\geq N$, $s_{a^+_k} = \ell(a^+_i) + n$, where $i$ satisfies that for every $j \leq $, $a^+_j$ does not contain $0^n$. 
\item For every $k > N$, $$\dfrac{1}{2^{\ell(a^+_N) + n}} \leq d({a^+_N}^{\prime \prime \prime}, a^+_k) \leq \dfrac{1}{2^{\ell(a^+_k)}}.$$ By Lemma \ref{splemma6}, $$\ell(a^+_N) \leq s_{a^+_k} \leq \ell(a^+_k) + 3n$$ for every $k > n$. 
\end{enumerate}
This implies that $$\mathop{\lim}\limits_{r \to \infty} m_r < \ell(a^+_N) + 3n.$$ Therefore, $\Sigma_a$ has specification. 
\end{proof}

An immediate consequence of Theorem \ref{specification4} is that for any $a \in \mathcal{E}$ such that $\Sigma_a$ does not have specification, the block $0^n$ occurs in $a$ infinitely many times. However, it will be shown that this condition is not sufficient.

\newtheorem{specification5}[subsection]{Theorem}
\begin{specification5}
Let $n \geq 2$ and $a \in \mathcal{E} \cap (1^n, 1^{n+1})$. If for every $r \in \mathbb{N}$, the canonical approximation from above $a^+_r$, satisfies that $$\dfrac{1}{2^{2\ell(a^+_r)}} < d({a^+_{r-1}}^{\prime}, a^+r) \leq \dfrac{1}{2^{\ell(a^+_r) + n}},$$ then $0^n$ occurs in $a$ infinitely many times and $(\Sigma_a, \sigma_a)$ has specification. \label{specification5}
\end{specification5}
\begin{proof}
Let $n \geq 2$ and $a \in \mathcal{E} \cap (1^n, 1^{n+1})$. Let $N$ such that for any $N \leq r$, $a^+_r$ does not contain $0^n$ and for any $r^{\prime} \leq N$, $a^+_{r^{\prime}}$ contains $0^n$. By Lemma \ref{splemma7} $s_{a^+_r} = \ell(a^+_{N}) + n$. Therefore $\mathop{\lim}\limits_{r \to \infty}m_r = \ell(a^+_r) + n$.
\end{proof}

Summing up Theorems \ref{specification1}, \ref{specification2}, \ref{specification3}, \ref{specification4}, \ref{specification5} and \ref{specification6} we obtain the proof of Theorem \ref{specification}. 

\vspace{1em}Observe that for the proposed families the sequence $\{r_i\}_{i=1}^{\infty} \subset \mathbb{N}$ stated in Theorem \ref{specification6} satisfies that $r_i = i$ for every $i$. Let $$S(\mathcal{E}) = \{a \in \mathcal{E}\mid \Sigma_a \hbox{\rm{ is transitive and has specification}}\},$$ and $NS(\mathcal{E}) = \mathcal{E} \setminus S(\mathcal{E}).$ We conjecture that $$\dim_{H} \mathcal{E} = \dim_{H}S(\mathcal{E}) = \dim_{H}NS(\mathcal{E}) = 1.$$ 

\subsubsection*{Non-Specification and The Thue-Morse Family}

The following Theorem will show that if the length of the defining words of the elements of the canonical approximation from above increases exponentially, then the limit sub-shift will not have specification.

\newtheorem{specification6}[subsection]{Theorem}
\begin{specification6}
Let $n \geq 2$ fixed. Let $a \in \mathcal{E}$ such that $a \in (1^n, 1^{n+1})$, $a \succ 11(01)^{\infty}$, $0^n$ occurs in $a$ infinitely many times and let $a_r^+$ be the approximation given by Theorem \ref{aproximation} $(2)$. If there exists an increasing sequence $\{r_i\}_{i=1}^{\infty} \subset \mathbb{N}$ and $R \in \mathbb{N}$ such that for every $r_i \geq R$ $a_{r_i}^+$ satisfies $$\ell(a_{r_{i-1}}^+ (\overline{{a_{r_{i-1}}^+}_1} \ldots \overline{{a_{r_{i-1}}^+}_{\ell(a_{r_{i-1}}^+)-1}}1)^{k_{r_i}}) \leq \ell(a_{r_i}^+)$$ and $$\dfrac{1}{2^{(k_{r_i}+1)\ell(a_{r_i}^+) + n}} \leq d({a_{r_i}^+}^{\prime \prime \prime}, a) \leq \dfrac{1}{2^{(k_{r_i}+1) \ell(a_{r_i}^+)}}$$ for some $k_{r_i} \geq 1$, then $(\Sigma_a, \sigma_a)$ does not have specification. \label{specification6}
\end{specification6}
\begin{proof}
Let $a \in  \mathcal{E}$ such that $a$ satisfies the hypothesis of the statement. Let $\{r_i\}_{i=1}^{\infty} \subset \mathbb{N}$ such that for every $r_i \geq R$, $a_{r_i}^+$ satisfies that $$\ell(a_{r_{i-1}}^+ (\overline{{a_{r_{i-1}}^+}_1} \ldots \overline{{a_{r_{i-1}}^+}_{\ell(a_{r_{i-1}}^+)-1}}1)^{k_{r_i}}) \leq \ell(a_{r_i}^+)$$ for some $k_{r_{i-1}} \geq 1$. By Proposition \ref{spremark1}, there exist $r_1$ such that $a^+_{r_1}$ does not contain $0^n$ and for every $r_i \geq r_1$  $$\dfrac{1}{2^{(k_1+1)\ell(a^+_{r_1}) + n}} \leq d({a^+_{r_1}}^{\prime \prime \prime}, a_{r_i}) \leq \dfrac{1}{2^{(k_1+1) \ell(a^+_{r_1})}}$$ for some $k_1 \in \mathbb{N}$. Then, by Lemma \ref{splemma5}, $k_1\ell(a^+_{r_1}) \leq s_{a^+_{r_i}}$ for every $i \geq 2$. Let $r_1 < r_2$, then by hypothesis $$\ell(a_{r_{1}}^+ (\overline{{a_{r_{1}}^+}_1} \ldots \overline{{a_{r_{1}}^+}_{\ell(a_{r_{1}}^+)-1}}1)^{k_{r_2}}) \leq \ell(a_{r_2}^+)$$ for some $k_{r_2} \geq 1$, and $\dfrac{1}{2^{(k_2+1)\ell(a^+_{r_1}) + n}} \leq d(a_{r_1},a_{r_2}) \leq \dfrac{1}{2^{(k_2+1)\ell(a^+_{r_1})}}$. Therefore, by Lemma \ref{splemma6} and Theorem \ref{transitivity} $$k_2k_1(s_{a^+_{r_1}}) \leq k_2k_1\ell(a^+_{r_1}) \leq k_2(s_{a^+_{r_1}}) \leq k_2(\ell(a^+_{r_1})) \leq s_{a^+_{r_2}}.$$ Inductively, $$(\mathop{\prod}\limits_{i=1}^{l}k_i)(s_{a^+_{r_1}}) \leq (\mathop{\prod}\limits_{i=1}^{l}k_i)\ell(a^+_{r_1}) \leq s_{a^+_{r_{l+1}}},$$ which implies that $\mathop{\lim}\limits_{r \to \infty} m_r$ is not bounded. Therefore $(\Sigma_a, \sigma_a)$ does not have specification.
\end{proof}

To illustrate Theorem \ref{specification6} a specific family of symmetric sub-shifts without specification will be constructed. Let $\omega$ be an irreducible word and $\omega \in (1^n, 1^{n+1})$ for $n \geq 2$ such that $0^n$ does not occur in $\omega$. Let $A$ be the block $[0^{n-1}1]$. Define $\omega^0 = \omega$, $\omega^1 = {\omega^0}^{\prime}A$ and for every $r \in \mathbb{N}$, $\omega^r = {\omega^{r-1}}^{\prime}A$. Note that by Proposition \ref{spremark1}, $\omega^r$ is an irreducible word for every $r \geq 0$. Then by Theorem \ref{exceptional1}, $\omega^* \in \mathcal{E}$. Let $\omega^* = \mathop{\lim}\limits_{r\to \infty} \omega^r$. Note that for any $r$ the last $n$ digits of $\omega^r$ are not $n$ consecutive ones. Therefore, the sequence $a^+_k = \omega^k 1$ approximates $\omega^*$ from above.  Note that the construction assures that $0^n$ occurs infinitely many times in $\omega^*$. Note also that for every $r \in \mathbb{N}$ $a_r^+$ the conditions of Theorem \ref{specification6} are satisfied. Therefore  $\Sigma_{\omega^*}$ does not have specification. 

\vspace{1em}We call this family \textit{the Thue-Morse non-specification family} and it is denoted by $T(\mathcal{E})$. Note that the growth of the specification number $s_{a^+_r}$ is exponential, that is $2^r(s_{\omega}) \leq s_{a^+_r}$ for every $r \in \mathbb{N}$. Observe that this can be done for every irreducible word $\omega$ and it can be modified considering $A_j = 0^j1$ with $j \in \{0, \ldots, n-1\}$, and considering the sequence 
\begin{itemize}
 \item[] $\omega^0 = \omega$;
 \item[] $\omega^1= \omega (\bar \omega_1 \ldots \bar \omega_{\ell(\omega)-1}1)k_1(A_j)_1$;
 \item[] and for any $r \geq 2$, $\omega^r = \omega_{r-1} (\bar \omega_1 \ldots \bar \omega_{\ell(\omega^r)-1}1)^{k_r}(A_j)_r$.
\end{itemize}

\section{Ergodic Properties of Symmetric Subshifts}
\label{resultsergodicity}

\noindent Given a shift space $(\Sigma, \sigma)$ we denote by $\mathbb{P}(\Sigma)$ the space of all probability measures $\mu$ defined on the Borel $\sigma$-algebra. We denote by $h_{\mu}(\sigma)$ \textit{the measure theoretical entropy of $\sigma$} or \textit{entropy of $\mu$}. The variational principle \cite[Theorem 8.6]{walters} states that $$h_{top}(\sigma) = \sup \{h_{\mu}(\sigma) \mid \mu \in \mathbb{P}(\Sigma)\}.$$ An invariant measure $\mu$ is called a \textit{measure of maximal entropy} if $h_{\mu}(\sigma) = h_{top}(\Sigma)$. We say that a shift transformation $\sigma$ is \textit{intrinsically ergodic} if $\sigma$ admits a unique measure of maximal entropy. Note that the expansivity of shift maps imply that such measures always exist \cite[Theorem 8.2, Theorem 8.7]{walters}. Furthermore, \cite[Proposition 7.7]{mane} states that if $(\Sigma, \sigma)$ is intrinsically ergodic then the unique measure of maximal entropy is ergodic. In our context, we have an explicit expression for such a measure, given by modifying the proof of \cite[Theorem 8.2, part \textit{ii}]{walters}. Given a symmetric sub-shift $\Sigma_a$, consider $$Per(n) = \{x \in \Sigma_a \mid \sigma^k(x) = x \hbox{\rm{ for some }} 1 \leq k \leq n \}.$$ Consider the sequence of measures given by $$\mu_n = \dfrac{1}{n} \mathop{\sum}\limits_{j=0}^{n-1} (\sigma^j)^*(\nu_n),$$ where $$\nu_n = \dfrac{1}{|Per(n)|}\mathop{\sum}\limits_{x \in Per(n)} \delta_x,$$ and $\mu$ is the weak star limit of a subsequence $\mu_{n_j}$. The main theorem of this section is showing that almost every symmetric sub-shift is intrinsically ergodic.

\newtheorem{almostevery}[subsection]{Theorem}
\begin{almostevery}
$(\Sigma_a, \sigma_a)$ is intrinsically ergodic for almost every $a \in (a^{*}, \frac{1}{2})$.\label{almostevery}
\end{almostevery}

In \cite{parry} Parry showed that transitive sub-shifts of finite type are intrinsically ergodic. Moreover, Weiss \cite{weiss1, weiss2} showed that every transitive sofic shift is intrinsically ergodic and Bowen \cite[p. 195]{bowen1} shown the same result for sub-shifts with the specification property.  As a consequence of Theorems \ref{transitivity} and \ref{sofic} we obtain the following theorem.

\newtheorem{iesft}[subsection]{Theorem}
\begin{iesft}
Let $a \succ 11(01)^{\infty}$. Then
\begin{enumerate}
 \item If $\Sigma_a$ is given by an irreducible word, then $\Sigma_a$ is intrinsically ergodic.
 \item If $\Sigma_a$ is given by a pre-periodic sequence and $a$ and $\Sigma_a$ is transitive then $\Sigma_a$ is intrinsically ergodic
\end{enumerate}
\label{iesft}
\end{iesft}

Furthermore, the families of sub-shifts described in Theorems \ref{specification1}, \ref{specification2}, \ref{specification3} are intrinsically ergodic.

\newtheorem{especificacion}[subsection]{Theorem}
\begin{especificacion}
Let $a \in \mathcal{E} \cap [11(01)^{\infty}, 1^{\infty}]$. Then:
\begin{enumerate}
\item If $0^n$ does not occur in $a$ then $(\Sigma_a, \sigma_a)$ is intrinsically ergodic;
\item If $0^n$ occurs in $a$ finite times then $(\Sigma_a, \sigma_a)$ is intrinsically ergodic;
\item If $0^n$ occurs in $a$ infinitely many times then $(\Sigma_a, \sigma_a)$ is intrinsically ergodic if $a_k^+$ satisfies Lemma \ref{splemma7}.
\end{enumerate}
\label{especification}
\end{especificacion}

It is not clear if it is possible to state under which conditions a general non-transitive sub-shift of finite type can be intrinsically ergodic. Actually, intrinsic ergodicity is not a consequence given by the transitivity of a sub-shift (see, e.g \cite{gurevich}, \cite{haydn}, \cite{petersen}). Nonetheless, it is conjectured that is possible to show that every symmetric sub-shift in two symbols is intrinsically ergodic. 

\newtheorem{ie1}[subsection]{Theorem}
\begin{ie1}
If $a \in \Sigma_2$ satisfies $a^* \prec a \prec 11(01)^{\infty}$ and $\Sigma_a$ is a shift of finite type, then $\Sigma_a$ is intrinsically ergodic. \label{ie1}
\end{ie1}
\begin{proof}
Let $A$ be the main component of $\Sigma_a$ given by Theorem \ref{component}, and $X = \Sigma_a \setminus A$. Let $\left[w\right]$ be a cylinder such that it is totally contained in $X$. Then $$\nu_n(\left[w\right]) = \dfrac{1}{|Per(n)|}\mathop{\sum}\limits_{x \in Per(n)} \delta_x = \dfrac{|Per(n) \cap \left[w\right]|}{|Per(n)|}.$$ Note that $\nu_n(\left[w\right]) \geq 0$ if and only if $w = u$ or $w = \bar u$, where $u = (01)^k$ for $k \geq \ell(\omega)$. In this case $$\nu_n(\left[w\right]) \leq \dfrac{2}{|Per(n)|}.$$ Observe that $\sigma^{-j}(\left[w \right])$ is a block of $(01)^{\infty}$, in this case $$\mu_n(\left[w\right]) = \dfrac{2}{|Per(n)|},$$ which tends to zero when $n$ tends to $\infty$ because $h_{top}(\Sigma_a) > 0$. This shows that the support of $\mu$ is $\Sigma_a \setminus X$.

\vspace{1em} Suppose that $\Sigma_a$ is not intrinsically ergodic. Then there exist $m$ such that $h_{\mu} = h_{m}$ and $\mu \neq m$. In \cite{bowen2, bowen1} it was shown that $\mu$ is ergodic. Then the support of $m$ has to be contained in $X$. Because the support of $m$ is a $\sigma$ invariant subset, it has to be $\{(01)^{\infty}, (10)^{\infty}\}$, then $h_m = 0$, which is a contradiction.
\end{proof}

This argument can be adapted to show the intrinsic ergodicity of any symmetric sub-shift given by a parameter $a$ lying on an entropy plateau, irrespective of whether $a \prec 11(01)^{\infty}$ or $a \succ 11(01)^{\infty}$. 

\newtheorem{ieplateaus1}[subsection]{Theorem}
\begin{ieplateaus1}
If $a$ is a parameter such that $a$ belongs to an entropy plateau $\left[\omega, \omega^{\prime \prime \prime}\right]$, then $(\Sigma_a, \sigma_a)$ is intrinsically ergodic.\label{ieplateaus1}
\end{ieplateaus1}
\begin{proof}
Suppose first that $a \succ 11(01)^{\infty}$ and $a$ belongs to an entropy plateau $[\omega, \omega^{\prime \prime \prime}]$. Then $\omega$ is an irreducible word, therefore, $(\Sigma_{\omega}, \sigma_{\omega})$ is a transitive sub-shift of finite type. By Theorem \ref{iesft}, $(\Sigma_{\omega}, \sigma_{\omega})$ is intrinsically ergodic. Let $\mu_{\omega}$ be the unique measure of maximal entropy of $(\Sigma_{\omega}, \sigma_{\omega})$. Note that we can extend $\mu_{\omega}$ to a measure $\mu_a$ on $\Sigma_a$ in the following way: $$\mu_a(B) = \mu_{\omega}(B \cap \Sigma_{\omega}).$$ Clearly, $\mu_a$ is an ergodic measure of maximal entropy for $\Sigma_a$, satisfying $$\hbox{\rm{supp}}(\mu_a) = \hbox{\rm{supp}}(\mu_{\omega}).$$ Suppose that there exist another measure $\nu_a$ of maximal entropy for $(\Sigma_a, \sigma_{a})$. By \cite[Theorem 8.7]{walters} we can consider $\nu_a$ as an ergodic measure for $\Sigma_a$. By \cite[Theorem 6.10]{walters} they are mutually singular. Let $A$ be the set such that $\mu_a(A) = 1 = \nu_a(\Sigma_a \setminus A)$. Recall that $$\mu_a(\hbox{\rm{supp}}(\mu_a) \cap \hbox{\rm{supp}}(\nu_a)) = \nu(\hbox{\rm{supp}}(\mu_a) \cap \hbox{\rm{supp}}(\nu_a)) = 0.$$ Note that for any measurable set $B \subset A$, $\nu_a(B)\log(\nu_a(B)) = 0$. Then $h_{\nu_a} = h_{top}(\sigma_a\mid_{\Sigma_a \setminus A})$. Note that $\Sigma_a \setminus A = \Sigma_a \setminus \Sigma_{\omega} \mod 0$. Then $h_{\nu_a} = h_{top}(\Sigma_a \setminus \Sigma_{\omega})$. Therefore, by Lemma \ref{controlentropy} $$h_{\nu_a} = h_{top}(\Sigma_a \setminus \Sigma_{\omega}) \leq \frac{1}{\ell(\omega)},$$ and by By Proposition \ref{powersoftwo} and Lemma \ref{entropyintervals} $$h_{\nu_a} < h_{\mu_a},$$  which contradicts that $\nu_a$ is a measure of maximal entropy. Therefore $(\Sigma_a, \sigma_a)$ is intrinsically ergodic.

Suppose now that $a \prec 11(01)^{\infty}$. Let $\omega$ be the left endpoint of the entropy plateau which $a$ belongs to. By Theorem \ref{ie1}, $\Sigma_{\omega}$ is intrinsically ergodic. Let $\mu_{\omega}$ be the unique measure of maximal entropy of $(\Sigma_{\omega}, \sigma_{\omega})$. Note that $\mu_{\omega}$ is supported in the main component of $(\Sigma_{\omega}, \sigma_{\omega})$, denoted by $A$. Let $$\mu_a(B) = \mu_{\omega}(B \cap A).$$ Using the same argument as in the transitive case, we can conclude that $\mu_a$ is the unique measure of maximal entropy for $(\Sigma_a, \sigma_a)$.
\end{proof}

As a consequence of Theorems \ref{iesft}, \ref{specification}, \ref{ieplateaus1} and the fact that $NS(\mathcal{E}) \subset \mathcal{E}$ Theorem \ref{almostevery} follows. Note that it is possible to construct examples of sub-shifts with two transitive components with the same topological entropy (see \cite{haydn}), then the argument shown in the proof of Theorem \ref{ieplateaus1} will not hold. 

\section{Final Remarks}
\label{despues}

\subsection*{Specification}
\noindent As was stated in Section \ref{specificationsection}, if a parameter $a \succ 11(01)^{\infty}$ is in the exceptional set $\mathcal{E}$ and $a \geq 11(01)^{\infty}$ then $\Sigma_a$ is a coded system. In \cite[Theorem B]{climenhaga}, sufficient conditions for a coded system to be intrinsically ergodic are given. Nonetheless, constructing the codes proposed \cite{climenhaga} is complicated. In \cite[p 571]{gurevich}, Gurevich states a general criterion to show intrinsic ergodicity for shifts that can be approximated from above. Let $a \in \mathcal{E}$ with $11(01)^{\infty} \prec a$. Consider $a^+_n$, given by Theorem \ref{aproximation}. Let $\rho_n = h_{top}(\Sigma_{a^+_n}) - h_{top}(\Sigma_a)$. Observe that $\rho_n \to 0$ when $n \to \infty$. Let $s_{a^+_n}$ be the specification number of $\Sigma_{a^+_n}$. Let $$R = -\mathop{\limsup}\limits_{n \to \infty} \dfrac{\log (\rho_n)}{n}$$ and $$S = \mathop{\limsup}\limits_{n \to \infty}\dfrac{s_{{a_n}^+}}{n}.$$ The criterion is summarised in the following Theorem.
 
\newtheorem{gurevichc}[subsection]{Theorem \cite[Theorem 1.1]{gurevich}}
\begin{gurevichc}
Let $a \in \mathcal{E}$ such that $\Sigma_a$ is transitive. If $$\dfrac{R}{(16+\gamma) h_{top}(\Sigma_a)} \geq S$$ for some $\gamma > 0$ then $\Sigma_a$ is intrinsically ergodic.\label{gurevichc}
\end{gurevichc}

Therefore, to use this criterion, it is necessary to have an exponential order of approximation of the entropies of the elements of the sequence given by Theorem \ref{aproximation}. A technical issue to deal with this limit is calculating the topological entropy for a sub-shift given by an irreducible word. However, if $a \in (1^k, 1^{k+1})$ for a fixed $k \geq 2$ it is possible to substitute $h_{top}(\Sigma_a)$ by $h_{top}(\Sigma_{1^{k+1}}) = \log \varphi^{k}$ where $\varphi^k$ is the $k$th multinacci number, that is $\varphi^k$ is the unique root of $x^k = \mathop{\sum}\limits_{n=0}^{k-1} x^n$ which lies in $(1,2)$. Indeed, it suffices to show that $$\rho_n \leq (\varphi^k)^{-(16+\gamma)\ell(a^+_n)} \leq (\varphi^k)^{-(16+\gamma)s_{{a_n}^+}},$$ for infinitely many $n \in \mathbb{N}$. 

\vspace{1em}It is difficult to describe the transitive components for $a \in \mathcal{E} \cap \left[a^*, 11(01)\right)$. It is known from \cite[Proposition 7.1]{bundfuss} that the transitive components of $a$ are coded, but the uniqueness of the main component it is still not known and if the main component will have the specification property.

\subsection*{Unique $\beta$-expansions for $\beta > 2$}

Recently, in \cite[Theorem 4.4]{simon} Baker has proven a generalisation of Theorem \ref{teo1}, placing particular emphasis on the differences when the number of symbols of the symmetric sub-shift associated to $\beta$ is odd or even. Then it is natural to ask if the results obtained so far can be extended for $\beta > 2$.

\section*{Acknowledgements} 

\noindent The author would like to thank Nikita Sidorov for his extensive support and Simon Baker for helpful discussions.

\end{document}